\documentclass{amsart}%
\usepackage{amsmath}
\usepackage{amssymb}
\usepackage{amsfonts}%
\usepackage{graphicx}
\providecommand{\U}[1]{\protect \rule{.1in}{.1in}}
\newtheorem{theorem}{Theorem}
\theoremstyle{plain}

\newtheorem{corollary}{Corollary}

\newtheorem{definition}{Definition}

\newtheorem{lemma}{Lemma}

\newtheorem{remark}{Remark}

\numberwithin{equation}{section}
\begin{document}
\title[{fractional type }multilinear commutators ]{{fractional type }multilinear commutators generated by fractional integral
with rough variable kernel and local Campanato functions on generalized
vanishing local Morrey spaces}
\author{F.Gurbuz}
\address{ANKARA UNIVERSITY, FACULTY OF SCIENCE, DEPARTMENT OF MATHEMATICS, TANDO\u{G}AN
06100, ANKARA, TURKEY }
\curraddr{}
\email{feritgurbuz84@hotmail.com}
\urladdr{}
\thanks{}
\thanks{}
\thanks{}
\date{}
\dedicatory{ }
\begin{abstract}
In this paper, we consider the boundedness of fractional type multilinear
commutators generated by fractional integral with rough variable kernel and
local Campanato functions on both generalized local {(central)} Morrey spaces
and generalized vanishing local Morrey spaces, under generic size conditions
which are satisfied by most of the operators in harmonic analysis, respectively.

\end{abstract}
\maketitle

\section{Introduction and main lemmas}

The classical Morrey spaces $L_{p,\lambda}$ were introduced by Morrey
\cite{Morrey} in 1938 to study the local behavior of solutions of second order
elliptic partial differential equations (PDEs). Later, there were many
applications of Morrey space to the Navier-Stokes equations (see
\cite{Mazzucato}), the Schr\"{o}dinger equations (see \cite{Ruiz}) and the
elliptic problems with discontinuous coefficients (see \cite{Caf, FazPalRag,
Pal}). The study of the operators of harmonic analysis in vanishing Morrey
space, in fact has been almost not touched. A version of the classical Morrey
space $L_{p,\lambda}({\mathbb{R}^{n}})$ where it is possible to approximate by
"nice" functions is the so called vanishing Morrey space $VL_{p,\lambda
}({\mathbb{R}^{n}})$ has been introduced by Vitanza in \cite{Vitanza1} and has
been applied there to obtain a regularity result for elliptic PDEs. This is a
subspace of functions in $L_{p,\lambda}({\mathbb{R}^{n}})$, which satisfies
the condition%
\[
\lim_{r\rightarrow0}\sup_{\underset{0<t<r}{x\in{\mathbb{R}^{n}}}}%
t^{-\frac{\lambda}{p}}\Vert f\Vert_{L_{p}(B(x,t))}=0.
\]
Later in \cite{Vitanza2} Vitanza has proved an existence theorem for a
Dirichlet problem, under weaker assumptions than in \cite{Miranda} and a
$W^{3,2}$ regularity result assuming that the partial derivatives of the
coefficients of the highest and lower order terms belong to vanishing Morrey
spaces depending on the dimension. Also Ragusa has proved a sufficient
condition for commutators of fractional integral operators to belong to
vanishing Morrey spaces $VL_{p,\lambda}({\mathbb{R}^{n}})$
(\cite{PerRagSamWall, RagusaJGlOpt}). For the properties and applications of
vanishing Morrey spaces, see also \cite{Cao-Chen}.

First of all, we recall some basic definitions and notations used in the paper.

Throughout the paper we assume that $x\in{\mathbb{R}^{n}}$ and $r>0$, and also
let $B(x,r)$ denote the open ball centered at $x$ of radius $r$, $B^{C}(x,r)$
denote its complement and $|B(x,r)|$ be the Lebesgue measure of the ball
$B(x,r)$, and $|B(x,r)|=v_{n}r^{n}$, where $v_{n}=|B(0,1)|$.

Recall that the concept of the generalized local (central) Morrey space
$LM_{p,\varphi}^{\{x_{0}\}}$ has been introduced in \cite{BGGS} and studied in
\cite{Gurbuz, Gurbuz3}.

\begin{definition}
\textbf{(Generalized local (central) Morrey space) }Let $\varphi(x,r)$ be a
positive measurable function on ${\mathbb{R}^{n}}\times(0,\infty)$ and $1\leq
p<\infty$. For any fixed $x_{0}\in{\mathbb{R}^{n}}$ we denote by
$LM_{p,\varphi}^{\{x_{0}\}}\equiv LM_{p,\varphi}^{\{x_{0}\}}({\mathbb{R}^{n}%
})$ the generalized local Morrey space, the space of all functions $f\in
L_{p}^{loc}({\mathbb{R}^{n}})$ with finite quasinorm
\[
\Vert f\Vert_{LM_{p,\varphi}^{\{x_{0}\}}}=\sup \limits_{r>0}\varphi
(x_{0},r)^{-1}\,|B(x_{0},r)|^{-\frac{1}{p}}\, \Vert f\Vert_{L_{p}(B(x_{0}%
,r))}<\infty.
\]

\end{definition}

According to this definition, we recover the local Morrey space $LL_{p,\lambda
}^{\{x_{0}\}}$ under the choice $\varphi(x_{0},r)=r^{\frac{\lambda-n}{p}}$:
\[
LL_{p,\lambda}^{\{x_{0}\}}=LM_{p,\varphi}^{\{x_{0}\}}\mid_{\varphi
(x_{0},r)=r^{\frac{\lambda-n}{p}}}.
\]
For the properties and applications of generalized local (central) Morrey
spaces $LM_{p,\varphi}^{\{x_{0}\}}$, see also \cite{BGGS, Gurbuz, Gurbuz3}.

\begin{definition}
\cite{BGGS, Gurbuz} Let $1\leq p<\infty$ and $0\leq \lambda<\frac{1}{n}$. A
local Campanato function $b\in L_{p}^{loc}\left(  {\mathbb{R}^{n}}\right)  $
is said to belong to the $LC_{p,\lambda}^{\left \{  x_{0}\right \}  }\left(
{\mathbb{R}^{n}}\right)  $ (local Campanato space), if%
\begin{equation}
\left \Vert b\right \Vert _{LC_{p,\lambda}^{\left \{  x_{0}\right \}  }}%
=\sup_{r>0}\left(  \frac{1}{\left \vert B\left(  x_{0},r\right)  \right \vert
^{1+\lambda p}}\int \limits_{B\left(  x_{0},r\right)  }\left \vert b\left(
y\right)  -b_{B\left(  x_{0},r\right)  }\right \vert ^{p}dy\right)  ^{\frac
{1}{p}}<\infty,\label{e51}%
\end{equation}
where%
\[
b_{B\left(  x_{0},r\right)  }=\frac{1}{\left \vert B\left(  x_{0},r\right)
\right \vert }\int \limits_{B\left(  x_{0},r\right)  }b\left(  y\right)  dy.
\]
Define%
\[
LC_{p,\lambda}^{\left \{  x_{0}\right \}  }\left(  {\mathbb{R}^{n}}\right)
=\left \{  b\in L_{p}^{loc}\left(  {\mathbb{R}^{n}}\right)  :\left \Vert
b\right \Vert _{LC_{p,\lambda}^{\left \{  x_{0}\right \}  }}<\infty \right \}  .
\]

\end{definition}

\begin{remark}
In \cite{LuYang1}, Lu and Yang have introduced the central BMO space
$CBMO_{p}({\mathbb{R}^{n}})=LC_{p,0}^{\{0\}}({\mathbb{R}^{n}})$. Note that
$BMO({\mathbb{R}^{n}})\subset \bigcap \limits_{p>1}LC_{p}^{\left \{
x_{0}\right \}  }({\mathbb{R}^{n}})$, $1\leq p<\infty$. Moreover, one can
imagine that the behavior of $LC_{p}^{\left \{  x_{0}\right \}  }({\mathbb{R}%
^{n}})$ may be quite different from that of $BMO({\mathbb{R}^{n}})$ (bounded
mean oscillation space{)}, since there is no analogy of the famous
John-Nirenberg inequality of $BMO({\mathbb{R}^{n}})$ for the space
$LC_{p}^{\left \{  x_{0}\right \}  }({\mathbb{R}^{n}})$.
\end{remark}

\begin{lemma}
\label{Lemma 4}\cite{BGGS, Gurbuz} Let $b$ be function in $LC_{p,\lambda
}^{\left \{  x_{0}\right \}  }\left(
%TCIMACRO{\U{211d} }%
%BeginExpansion
\mathbb{R}
%EndExpansion
^{n}\right)  $, $1\leq p<\infty$, $0\leq \lambda<\frac{1}{n}$ and $r_{1}$,
$r_{2}>0$. Then%
\begin{equation}
\left(  \frac{1}{\left \vert B\left(  x_{0},r_{1}\right)  \right \vert
^{1+\lambda p}}%
%TCIMACRO{\dint \limits_{B\left(  x_{0},r_{1}\right)  }}%
%BeginExpansion
{\displaystyle \int \limits_{B\left(  x_{0},r_{1}\right)  }}
%EndExpansion
\left \vert b\left(  y\right)  -b_{B\left(  x_{0},r_{2}\right)  }\right \vert
^{p}dy\right)  ^{\frac{1}{p}}\leq C\left(  1+\ln \frac{r_{1}}{r_{2}}\right)
\left \Vert b\right \Vert _{LC_{p,\lambda}^{\left \{  x_{0}\right \}  }},\label{a}%
\end{equation}
where $C>0$ is independent of $b$, $r_{1}$ and $r_{2}$.

From this inequality $\left(  \text{\ref{a}}\right)  $, we have%
\begin{equation}
\left \vert b_{B\left(  x_{0},r_{1}\right)  }-b_{B\left(  x_{0},r_{2}\right)
}\right \vert \leq C\left(  1+\ln \frac{r_{1}}{r_{2}}\right)  \left \vert
B\left(  x_{0},r_{1}\right)  \right \vert ^{\lambda}\left \Vert b\right \Vert
_{LC_{p,\lambda}^{\left \{  x_{0}\right \}  }},\label{b}%
\end{equation}
and it is easy to see that%
\begin{equation}
\left \Vert b-\left(  b\right)  _{B}\right \Vert _{L_{p}\left(  B\right)  }\leq
C\left(  1+\ln \frac{r_{1}}{r_{2}}\right)  r^{\frac{n}{p}+n\lambda}\left \Vert
b\right \Vert _{LC_{p,\lambda}^{\left \{  x_{0}\right \}  }}.\label{c}%
\end{equation}

\end{lemma}

For brevity, in the sequel we use the following notation%
\[
\mathfrak{M}_{p,\varphi}\left(  f;x_{0},r\right)  :=\frac{|B(x_{0}%
,r)|^{-\frac{1}{p}}\, \Vert f\Vert_{L_{p}(B(x_{0},r))}}{\varphi(x_{0},r)}.
\]
Extending the definition of the vanishing Morrey spaces to the case of the
generalized local (central) Morrey spaces, we introduce the following definition.

\begin{definition}
\textbf{(generalized vanishing local Morrey space) }The generalized vanishing
local Morrey space $VLM_{p,\varphi}^{\left \{  x_{0}\right \}  }({\mathbb{R}%
^{n}})$ is defined as the spaces of functions $f\in LM_{p,\varphi}^{\{x_{0}%
\}}({\mathbb{R}^{n}})$ such that%
\[
\lim \limits_{r\rightarrow0}\mathfrak{M}_{p,\varphi}\left(  f;x_{0},r\right)
=0.
\]

\end{definition}

Throughout the sequel we assume that%
\begin{equation}
\lim_{r\rightarrow0}\frac{1}{\varphi(x_{0},r)}=0,\label{2}%
\end{equation}
and%
\begin{equation}
\sup_{0<r<\infty}\frac{1}{\varphi(x_{0},r)}<\infty,\label{3}%
\end{equation}
which make the spaces $VLM_{p,\varphi}^{\left \{  x_{0}\right \}  }%
({\mathbb{R}^{n}})$ non-trivial, since bounded functions with compact support
belong to this space. The space $VLM_{p,\varphi}^{\left \{  x_{0}\right \}
}({\mathbb{R}^{n}})$ is a Banach space with respect to the norm
\[
\Vert f\Vert_{VLM_{p,\varphi}^{\left \{  x_{0}\right \}  }}\equiv \Vert
f\Vert_{LM_{p,\varphi}^{\{x_{0}\}}}=\sup \limits_{r>0}\mathfrak{M}_{p,\varphi
}\left(  f;x_{0},r\right)  .
\]
The space $VLM_{p,\varphi}^{\left \{  x_{0}\right \}  }({\mathbb{R}^{n}})$ is a
closed subspace of the Banach space $LM_{p,\varphi}^{\left \{  x_{0}\right \}
}({\mathbb{R}^{n}})$, which may be shown by standard means.

Suppose that $S^{n-1}$ is the unit sphere on ${\mathbb{R}^{n}}$ $(n\geq2)$
equipped with the normalized Lebesgue measure $d\sigma \left(  x^{\prime
}\right)  $. A function $\Omega \left(  x,z\right)  $ defined on ${\mathbb{R}%
^{n}\times \mathbb{R}^{n}}$ is said to belong to $L_{\infty}({\mathbb{R}^{n}%
})\times L_{s}(S^{n-1})$ for $s>1$, if $\Omega$ satisfies the following conditions:

for any $x$, $z\in{\mathbb{R}^{n}}$ and $\lambda>0$,%

\[
\Omega(x,\lambda z)=\Omega(x,z);
\]%
\[
\Vert \Omega \Vert_{L_{\infty}({\mathbb{R}^{n}})\times L_{s}(S^{n-1})}%
:=\sup_{x\in{\mathbb{R}^{n}}}\left(
%TCIMACRO{\dint \limits_{S^{n-1}}}%
%BeginExpansion
{\displaystyle \int \limits_{S^{n-1}}}
%EndExpansion
\left \vert \Omega(x,z^{\prime})\right \vert ^{s}d\sigma \left(  z^{\prime
}\right)  \right)  ^{1/s}<\infty,
\]
where $z^{\prime}=z/\left \vert z\right \vert $, for any $z\in{\mathbb{R}^{n}%
}\setminus \{0\}$.

Suppose that $T_{\Omega,\alpha}$, $\alpha \in \left(  0,n\right)  $ represents a
linear or a sublinear operator with rough variable kernel, which satisfies
that for any $f\in S({\mathbb{R}^{n}})$ with compact support and $x\notin
suppf$
\begin{equation}
|T_{\Omega,\alpha}f(x)|\leq c_{0}\int \limits_{{\mathbb{R}^{n}}}\frac
{|\Omega(x,x-y)|}{|x-y|^{n-\alpha}}\,|f(y)|\,dy,\label{e1}%
\end{equation}
where $c_{0}$ is independent of $f$ and $x$. We point out that the condition
(\ref{e1}) in the case of $\Omega \equiv1$, $\alpha=0$ has been introduced by
Soria and Weiss in \cite{SW}. The condition (\ref{e1}) is satisfied by many
interesting operators in harmonic analysis, such as the fractional
Marcinkiewicz operator, fractional maximal operator, fractional integral
operator (Riesz potential) and so on (see \cite{LLY}, \cite{SW} for details).

Then, the fractional integral operator with rough variable kernel
$\overline{T}_{\Omega,\alpha}$ is defined by%

\[
\overline{T}_{\Omega,\alpha}f(x)=\int \limits_{{\mathbb{R}^{n}}}\frac
{\Omega(x,x-y)}{|x-y|^{n-\alpha}}f(y)dy\qquad0<\alpha<n,
\]
and a related fractional maximal operator with rough variable kernel
$M_{\Omega,\alpha}$ is given by%

\[
M_{\Omega,\alpha}f(x)=\sup_{t>0}|B(x,t)|^{-1+\frac{\alpha}{n}}\int
\limits_{B(x,t)}\left \vert \Omega(x,x-y)\right \vert |f(y)|dy\qquad0<\alpha<n,
\]
where $\Omega \in L_{\infty}({\mathbb{R}^{n}})\times L_{s}(S^{n-1})$, $s>1$, is
homogeneous of degree zero with respect to the second variable $y$ on
${\mathbb{R}^{n}}$ and these operators also satisfy condition (\ref{e1}).

If $\alpha=0$, then $\overline{T}_{\Omega,\alpha}$ and $M_{\Omega,\alpha}$
becomes the singular integral operator and the Hardy-Littlewood maximal
operator by with rough variable kernels, respectively. It is obvious that when
$\Omega \equiv1$, $\overline{T}_{1,\alpha}\equiv \overline{T}_{\alpha}$ and
$M_{1,\alpha}\equiv M_{\alpha}$ are the fractional integral operator (Riesz
potential) and the fractional maximal operator, respectively.

In recent years, the mapping properties of $\overline{T}_{\Omega,\alpha}$ on
some kinds of function spaces have been studied in many papers (see
\cite{Aos1, Aos2, Chen, Muckenhoupt, Wang, Zhang} for details). In particular,
the boundedness of $\overline{T}_{\Omega,\alpha}$ for rough variable kernel in
Lebesgue spaces has been obtained by Muckenhoupt and Wheeden
\cite{Muckenhoupt} as follows:

\begin{lemma}
\label{Lemma1}\cite{Muckenhoupt} Let $0<\alpha<n$, $1<p<\frac{n}{\alpha}$ and
$\frac{1}{q}=\frac{1}{p}-\frac{\alpha}{n}$. If $\Omega \in L_{\infty
}({\mathbb{R}^{n}})\times L_{s}(S^{n-1})$, $s>\frac{n}{n-\alpha}$, for
$s^{\prime}\leq p$ or $q<s$, then we have%
\[
\left \Vert \overline{T}_{\Omega,\alpha}f\right \Vert _{L_{q}\left(
{\mathbb{R}^{n}}\right)  }\leq C\Vert \Omega \Vert_{L_{\infty}({\mathbb{R}^{n}%
})\times L_{s}(S^{n-1})}\left \Vert f\right \Vert _{L_{p}\left(  {\mathbb{R}%
^{n}}\right)  }.
\]

\end{lemma}

\begin{corollary}
Under the assumptions of Lemma \ref{Lemma1}, the operator $M_{\Omega,\alpha}$
from $L_{p}({\mathbb{R}^{n}})$ to $L_{q}({\mathbb{R}^{n}})$ is bounded .
Moreover, we have%
\[
\left \Vert M_{\Omega,\alpha}f\right \Vert _{L_{q}\left(  {\mathbb{R}^{n}%
}\right)  }\leq C\Vert \Omega \Vert_{L_{\infty}({\mathbb{R}^{n}})\times
L_{s}(S^{n-1})}\left \Vert f\right \Vert _{L_{p}\left(  {\mathbb{R}^{n}}\right)
}.
\]

\end{corollary}

\begin{proof}
Set%
\[
\widetilde{T}_{\left \vert \Omega \right \vert ,\alpha}\left(  \left \vert
f\right \vert \right)  (x)=\int \limits_{{\mathbb{R}^{n}}}\frac{\left \vert
\Omega(x,x-y)\right \vert }{|x-y|^{n-\alpha}}\left \vert f(y)\right \vert
dy,\qquad0<\alpha<n,
\]
where $\Omega \in L_{\infty}({\mathbb{R}^{n}})\times L_{s}(S^{n-1})\left(
s>1\right)  $ is homogeneous of degree zero with respect to the second
variable $y$ on ${\mathbb{R}^{n}}$. It is easy to see that, $\widetilde
{T}_{\left \vert \Omega \right \vert ,\alpha}$ satisfies Lemma \ref{Lemma1}. On
the other hand, for any $t>0$, we have
\begin{align*}
\widetilde{T}_{\left \vert \Omega \right \vert ,\alpha}\left(  \left \vert
f\right \vert \right)  (x)  & \geq%
%TCIMACRO{\dint \limits_{B\left(  x,t\right)  }}%
%BeginExpansion
{\displaystyle \int \limits_{B\left(  x,t\right)  }}
%EndExpansion
\frac{\left \vert \Omega(x,x-y)\right \vert }{|x-y|^{n-\alpha}}\left \vert
f(y)\right \vert dy\\
& \geq \frac{1}{t^{n-\alpha}}%
%TCIMACRO{\dint \limits_{B\left(  x,t\right)  }}%
%BeginExpansion
{\displaystyle \int \limits_{B\left(  x,t\right)  }}
%EndExpansion
\left \vert \Omega(x,x-y)\right \vert \left \vert f(y)\right \vert dy.
\end{align*}
Then by taking supremum for $t>0$, we get%
\[
M_{\Omega,\alpha}f\left(  x\right)  \leq C_{n,\alpha}^{-1}\widetilde
{T}_{\left \vert \Omega \right \vert ,\alpha}\left(  \left \vert f\right \vert
\right)  (x)\qquad C_{n,\alpha}=\left \vert B\left(  0,1\right)  \right \vert
^{\frac{n-\alpha}{n}}.
\]

\end{proof}

On the other hand, commutators of linear operators take important roles in
harmonic analysis and related topics (see \cite{Chanillo, Coifman1, Coifman2,
Gurbuz, Gurbuz2, Gurbuz3, Janson, Palus, Shi}). There are two major reasons
for considering the problem of commutators. The first one is that the
boundedness of commutators can produce some characterizations of function
spaces (see \cite{BGGS, Chanillo, Gurbuz, Gurbuz1, Gurbuz2, Gurbuz3, Janson,
Palus, Shi}). The other one is that the theory of commutators plays an
important role in the study of the regularity of solutions to elliptic and
parabolic PDEs of the second order (see \cite{ChFraL1, ChFraL2, FazPalRag,
Softova}).

Let $b_{i}\left(  i=1,\ldots,m\right)  $ be locally integrable functions on
${\mathbb{R}^{n}}$, then the multilinear commutators generated by fractional
integral and maximal operator with rough variable kernel and $\overrightarrow
{b}=\left(  b_{1},\ldots,b_{m}\right)  $ are given as follows, respectively:%
\[
\lbrack \overrightarrow{b},\overline{T}_{\Omega,\alpha}]f\left(  x\right)  =%
%TCIMACRO{\dint \limits_{{\mathbb{R}^{n}}}}%
%BeginExpansion
{\displaystyle \int \limits_{{\mathbb{R}^{n}}}}
%EndExpansion%
%TCIMACRO{\dprod \limits_{i=1}^{m}}%
%BeginExpansion
{\displaystyle \prod \limits_{i=1}^{m}}
%EndExpansion
\left[  b_{i}\left(  x\right)  -b_{i}\left(  y\right)  \right]  \frac
{\Omega(x,x-y)}{|x-y|^{n-\alpha}}f\left(  y\right)  dy,\qquad0<\alpha<n,
\]%
\[
M_{\Omega,\overrightarrow{b},\alpha}f\left(  x\right)  =\sup_{t>0}%
|B(x,t)|^{-1+\frac{\alpha}{n}}%
%TCIMACRO{\dint \limits_{B(x,t)}}%
%BeginExpansion
{\displaystyle \int \limits_{B(x,t)}}
%EndExpansion%
%TCIMACRO{\dprod \limits_{i=1}^{m}}%
%BeginExpansion
{\displaystyle \prod \limits_{i=1}^{m}}
%EndExpansion
\left[  b_{i}\left(  x\right)  -b_{i}\left(  y\right)  \right]  \left \vert
\Omega(x,x-y)\right \vert |f(y)|dy,\qquad0<\alpha<n.
\]

For $m=1$, $[\overrightarrow{b},\overline{T}_{\Omega,\alpha}]$ and
$M_{\Omega,\overrightarrow{b},\alpha}$ are obviously the commutator operators
of $\overline{T}_{\Omega,\alpha}$ and $M_{\Omega,\alpha}$,%
\begin{align*}
\lbrack b,\overline{T}_{\Omega,\alpha}]f(x)  & \equiv b(x)\overline{T}%
_{\Omega,\alpha}f(x)-\overline{T}_{\Omega,\alpha}(bf)(x)\\
& =\int \limits_{{\mathbb{R}^{n}}}[b(x)-b(y)]\frac{\Omega(x,x-y)}%
{|x-y|^{n-\alpha}}f(y)dy,
\end{align*}
and
\begin{align*}
M_{\Omega,b,\alpha}\left(  f\right)  (x)  & \equiv b\left(  x\right)
M_{\Omega,\alpha}f\left(  x\right)  -M_{\Omega,\alpha}\left(  bf\right)
\left(  x\right) \\
& =\sup_{t>0}|B(x,t)|^{-1+\frac{\alpha}{n}}\int \limits_{B(x,t)}\left \vert
b\left(  x\right)  -b\left(  y\right)  \right \vert \left \vert \Omega
(x,x-y)\right \vert |f(y)|dy,
\end{align*}
where $0<\alpha<n$ and $f$ is a suitable function.

In \cite{Mo}, the authors obtain the boundedness for the multilinear
commutators generated by singular integral operators with rough variable
kernel and local Campanato functions on generalized local Morrey spaces.

Inspired by \cite{Mo}, in this paper we give local {Campanato space} estimates
for fractional type multilinear commutators with rough variable kernel on both
generalized local Morrey spaces and generalized vanishing local Morrey spaces,
respectively. But, the techniques and non-trivial estimates which have been
used in the proofs of our main results are quite different from \cite{Mo}. For
example, using inequality about the weighted Hardy operator $H_{w}$ in
\cite{Mo}, in this paper we will only use the following relationship between
essential supremum and essential infimum%
\begin{equation}
\left(  \operatorname*{essinf}\limits_{x\in E}f\left(  x\right)  \right)
^{-1}=\operatorname*{esssup}\limits_{x\in E}\frac{1}{f\left(  x\right)
},\label{5}%
\end{equation}
where $f$ is any real-valued nonnegative function and measurable on $E$ (see
\cite{Wheeden-Zygmund}, page 143).

We first need some lemmas (our main lemmas) which are used in the proof of the
main results. These lemmas with their proofs can be formulated as follows, respectively:

\begin{lemma}
\label{lemma1}Suppose that $x_{0}\in{\mathbb{R}^{n}}$, $\Omega \in L_{\infty
}({\mathbb{R}^{n}})\times L_{s}(S^{n-1})$, $s>1$, is homogeneous of degree
zero with respect to the second variable $y$ on ${\mathbb{R}^{n}}$. Let
$0<\alpha<n$, $1<p<\frac{n}{\alpha}$, $\frac{1}{q}=\frac{1}{p}-\frac{\alpha
}{n}$. Let $T_{\Omega,\alpha}$ be a sublinear operator satisfying condition
(\ref{e1}), bounded from $L_{p}({\mathbb{R}^{n}})$ to $L_{q}({\mathbb{R}^{n}%
})$.

If $p>1$ and $s^{\prime}\leq p$, then the inequality
\begin{equation}
\left \Vert T_{\Omega,\alpha}f\right \Vert _{L_{q}\left(  B\left(
x_{0},r\right)  \right)  }\lesssim r^{\frac{n}{q}}\int \limits_{2r}^{\infty
}t^{-\frac{n}{q}-1}\left \Vert f\right \Vert _{L_{p}\left(  B\left(
x_{0},t\right)  \right)  }dt\label{40}%
\end{equation}
holds for any ball $B\left(  x_{0},r\right)  $ and for all $f\in L_{p}%
^{loc}\left(  {\mathbb{R}^{n}}\right)  $.

If $p>1$ and $q<s$, then the inequality%
\[
\left \Vert T_{\Omega,\alpha}f\right \Vert _{L_{q}\left(  B\left(
x_{0},r\right)  \right)  }\lesssim r^{\frac{n}{q}-\frac{n}{s}}\int
\limits_{2r}^{\infty}t^{\frac{n}{s}-\frac{n}{q}-1}\left \Vert f\right \Vert
_{L_{p}\left(  B\left(  x_{0},t\right)  \right)  }dt
\]
holds for any ball $B\left(  x_{0},r\right)  $ and for all $f\in L_{p}%
^{loc}\left(  {\mathbb{R}^{n}}\right)  $.
\end{lemma}

\begin{proof}
Let $0<\alpha<n$, $1\leq s^{\prime}<p<\frac{n}{\alpha}$ and $\frac{1}{q}%
=\frac{1}{p}-\frac{\alpha}{n}$. Set $B=B\left(  x_{0},r\right)  $ for the ball
centered at $x_{0}$ and of radius $r$ and $2B=B\left(  x_{0},2r\right)  $. We
represent $f$ as%
\begin{equation}
f=f_{1}+f_{2},\qquad \text{\ }f_{1}\left(  y\right)  =f\left(  y\right)
\chi_{2B}\left(  y\right)  ,\qquad \text{\ }f_{2}\left(  y\right)  =f\left(
y\right)  \chi_{\left(  2B\right)  ^{C}}\left(  y\right)  ,\qquad
r>0\label{e39}%
\end{equation}
and have%
\[
\left \Vert T_{\Omega,\alpha}f\right \Vert _{L_{q}\left(  B\right)  }%
\leq \left \Vert T_{\Omega,\alpha}f_{1}\right \Vert _{L_{q}\left(  B\right)
}+\left \Vert T_{\Omega,\alpha}f_{2}\right \Vert _{L_{q}\left(  B\right)  }.
\]

Since $f_{1}\in L_{p}\left(  \mathbb{R}^{n}\right)  $, $T_{\Omega,\alpha}%
f_{1}\in L_{q}\left(  \mathbb{R}^{n}\right)  $ and by the boundedness of
$T_{\Omega,\alpha}$ from $L_{p}({\mathbb{R}^{n}})$ to $L_{q}({\mathbb{R}^{n}%
})$ (see Lemma \ref{Lemma1}) it follows that:%
\[
\left \Vert T_{\Omega,\alpha}f_{1}\right \Vert _{L_{q}\left(  B\right)  }%
\leq \left \Vert T_{\Omega,\alpha}f_{1}\right \Vert _{L_{q}\left(
%TCIMACRO{\U{211d} }%
%BeginExpansion
\mathbb{R}
%EndExpansion
^{n}\right)  }\leq C\left \Vert f_{1}\right \Vert _{L_{p}\left(
%TCIMACRO{\U{211d} }%
%BeginExpansion
\mathbb{R}
%EndExpansion
^{n}\right)  }=C\left \Vert f\right \Vert _{L_{p}\left(  2B\right)  },
\]
where constant $C>0$ is independent of $f$.

It is clear that $x\in B$, $y\in \left(  2B\right)  ^{C}$ implies \ $\frac
{1}{2}\left \vert x_{0}-y\right \vert \leq \left \vert x-y\right \vert \leq \frac
{3}{2}\left \vert x_{0}-y\right \vert $. Then we get%
\[
\left \vert T_{\Omega,\alpha}f_{2}\left(  x\right)  \right \vert \leq
2^{n-\alpha}c_{1}\int \limits_{\left(  2B\right)  ^{C}}\frac{\left \vert
f\left(  y\right)  \right \vert \left \vert \Omega(x,\text{ }x-y)\right \vert
}{\left \vert x_{0}-y\right \vert ^{n-\alpha}}dy.
\]
By Fubini's theorem, we have%
\begin{align*}
\int \limits_{\left(  2B\right)  ^{C}}\frac{\left \vert f\left(  y\right)
\right \vert \left \vert \Omega(x,x-y)\right \vert }{\left \vert x_{0}%
-y\right \vert ^{n-\alpha}}dy  & \approx \int \limits_{\left(  2B\right)  ^{C}%
}\left \vert f\left(  y\right)  \right \vert \left \vert \Omega(x,x-y)\right \vert
\int \limits_{\left \vert x_{0}-y\right \vert }^{\infty}\frac{dt}{t^{n+1-\alpha}%
}dy\\
& \approx \int \limits_{2r}^{\infty}\int \limits_{2r\leq \left \vert x_{0}%
-y\right \vert \leq t}\left \vert f\left(  y\right)  \right \vert \left \vert
\Omega(x,x-y)\right \vert dy\frac{dt}{t^{n+1-\alpha}}\\
& \lesssim \int \limits_{2r}^{\infty}\int \limits_{B\left(  x_{0},t\right)
}\left \vert f\left(  y\right)  \right \vert \left \vert \Omega(x,x-y)\right \vert
dy\frac{dt}{t^{n+1-\alpha}}.
\end{align*}

Applying H\"{o}lder's inequality, we get%
\begin{align}
& \int \limits_{\left(  2B\right)  ^{C}}\frac{\left \vert f\left(  y\right)
\right \vert \left \vert \Omega(x,x-y)\right \vert }{\left \vert x_{0}%
-y\right \vert ^{n-\alpha}}dy\nonumber \\
& \lesssim \int \limits_{2r}^{\infty}\left \Vert f\right \Vert _{L_{p}\left(
B\left(  x_{0},t\right)  \right)  }\left \Vert \Omega \left(  x,x-\cdot \right)
\right \Vert _{L_{s}\left(  B\left(  x_{0},t\right)  \right)  }\left \vert
B\left(  x_{0},t\right)  \right \vert ^{1-\frac{1}{p}-\frac{1}{s}}\frac
{dt}{t^{n+1-\alpha}}.\label{e310}%
\end{align}
For $x\in B\left(  x_{0},t\right)  $, notice that $\Omega$ is homogeneous of
degree zero with respect to the second variable $y$ on ${\mathbb{R}^{n}}$ and
$\Omega \in L_{\infty}({\mathbb{R}^{n}})\times L_{s}(S^{n-1})$, $s>1$. Then, we
obtain%
\begin{align}
\left(  \int \limits_{B\left(  x_{0},t\right)  }\left \vert \Omega
(x,x-y)\right \vert ^{s}dy\right)  ^{\frac{1}{s}}  & =\left(  \int
\limits_{B\left(  x-x_{0},t\right)  }\left \vert \Omega \left(  x,z\right)
\right \vert ^{s}dz\right)  ^{\frac{1}{s}}\nonumber \\
& \leq \left(  \int \limits_{B\left(  0,t+\left \vert x-x_{0}\right \vert \right)
}\left \vert \Omega \left(  x,z\right)  \right \vert ^{s}dz\right)  ^{\frac{1}%
{s}}\nonumber \\
& \leq \left(  \int \limits_{B\left(  0,2t\right)  }\left \vert \Omega \left(
x,z\right)  \right \vert ^{s}dz\right)  ^{\frac{1}{s}}\nonumber \\
& =\left(  \int \limits_{S^{n-1}}\int \limits_{0}^{2t}\left \vert \Omega \left(
x,z^{\prime}\right)  \right \vert ^{s}d\sigma \left(  z^{\prime}\right)
r^{n-1}dr\right)  ^{\frac{1}{s}}\nonumber \\
& =C\left \Vert \Omega \right \Vert _{L_{\infty}({\mathbb{R}^{n}})\times
L_{s}(S^{n-1})}\left \vert B\left(  x_{0},2t\right)  \right \vert ^{\frac{1}{s}%
}.\label{e311}%
\end{align}
Thus, by (\ref{e311}), it follows that:%
\[
\left \vert T_{\Omega,\alpha}f_{2}\left(  x\right)  \right \vert \lesssim
\int \limits_{2r}^{\infty}\left \Vert f\right \Vert _{L_{p}\left(  B\left(
x_{0},t\right)  \right)  }\frac{dt}{t^{\frac{n}{q}+1}}.
\]

Moreover, for all $p\in \left(  1,\infty \right)  $ the inequality%
\begin{equation}
\left \Vert T_{\Omega,\alpha}f_{2}\right \Vert _{L_{q}\left(  B\right)
}\lesssim r^{\frac{n}{q}}\int \limits_{2r}^{\infty}\left \Vert f\right \Vert
_{L_{p}\left(  B\left(  x_{0},t\right)  \right)  }\frac{dt}{t^{\frac{n}{q}+1}%
}\label{10}%
\end{equation}

is valid. Thus, we obtain%
\[
\left \Vert T_{\Omega,\alpha}f\right \Vert _{L_{q}\left(  B\right)  }%
\lesssim \left \Vert f\right \Vert _{L_{p}\left(  2B\right)  }+r^{\frac{n}{q}%
}\int \limits_{2r}^{\infty}\left \Vert f\right \Vert _{L_{p}\left(  B\left(
x_{0},t\right)  \right)  }\frac{dt}{t^{\frac{n}{q}+1}}.
\]

On the other hand, we have%
\begin{align}
\left \Vert f\right \Vert _{L_{p}\left(  2B\right)  }  & \approx r^{\frac{n}{q}%
}\left \Vert f\right \Vert _{L_{p}\left(  2B\right)  }\int \limits_{2r}^{\infty
}\frac{dt}{t^{\frac{n}{q}+1}}\nonumber \\
& \leq r^{\frac{n}{q}}\int \limits_{2r}^{\infty}\left \Vert f\right \Vert
_{L_{p}\left(  B\left(  x_{0},t\right)  \right)  }\frac{dt}{t^{\frac{n}{q}+1}%
}.\label{e313}%
\end{align}
By combining the above inequalities, we obtain%
\[
\left \Vert T_{\Omega,\alpha}f\right \Vert _{L_{q}\left(  B\right)  }\lesssim
r^{\frac{n}{q}}\int \limits_{2r}^{\infty}t^{-\frac{n}{q}-1}\left \Vert
f\right \Vert _{L_{p}\left(  B\left(  x_{0},t\right)  \right)  }dt.
\]
For the case of $1<q<s$, we can also use the same method, so we omit the
details. This completes the proof of Lemma \ref{lemma1}.
\end{proof}

\begin{lemma}
\label{lemma2}Suppose that $x_{0}\in{\mathbb{R}^{n}}$, $\Omega \in L_{\infty
}({\mathbb{R}^{n}})\times L_{s}(S^{n-1})$, $s>1$, is homogeneous of degree
zero with respect to the second variable $y$ on ${\mathbb{R}^{n}}$. Let
$T_{\Omega,\alpha}$ be a linear operator satisfying condition (\ref{e1}). Let
also $0<\alpha<n$ and $1<q,q_{1},p_{i},p<\frac{n}{\alpha}$ with $\frac{1}%
{q}=\sum \limits_{i=1}^{m}\frac{1}{p_{i}}+\frac{1}{p}$, $\frac{1}{q_{1}}%
=\frac{1}{q}-\frac{\alpha}{n}$ and $\overrightarrow{b}\in LC_{p_{i}%
,\lambda_{i}}^{\left \{  x_{0}\right \}  }({\mathbb{R}^{n}})$ for $0\leq
\lambda_{i}<\frac{1}{n}$, $i=1,\ldots,m$. Then, for $s^{\prime}\leq q$ the
inequality
\begin{equation}
\Vert \lbrack \overrightarrow{b},T_{\Omega,\alpha}]f\Vert_{L_{q_{1}}%
(B(x_{0},r))}\lesssim%
%TCIMACRO{\dprod \limits_{i=1}^{m}}%
%BeginExpansion
{\displaystyle \prod \limits_{i=1}^{m}}
%EndExpansion
\Vert \overrightarrow{b}\Vert_{LC_{p_{i},\lambda_{i}}^{\left \{  x_{0}\right \}
}}r^{\frac{n}{q_{1}}}%
%TCIMACRO{\dint \limits_{2r}^{\infty}}%
%BeginExpansion
{\displaystyle \int \limits_{2r}^{\infty}}
%EndExpansion
\left(  1+\ln \frac{t}{r}\right)  ^{m}t^{n\left(  -\frac{1}{q_{1}}+\left(
%TCIMACRO{\dsum \limits_{i=1}^{m}}%
%BeginExpansion
{\displaystyle \sum \limits_{i=1}^{m}}
%EndExpansion
\lambda_{i}+%
%TCIMACRO{\dsum \limits_{i=1}^{m}}%
%BeginExpansion
{\displaystyle \sum \limits_{i=1}^{m}}
%EndExpansion
\frac{1}{p_{i}}\right)  \right)  -1}\Vert f\Vert_{L_{p}(B(x_{0},t))}%
dt\label{200}%
\end{equation}
holds for any ball $B(x_{0},r)$ and for all $f\in L_{p}^{loc}({\mathbb{R}^{n}%
})$. Also, for $q_{1}<s$ the inequality%
\[
\Vert \lbrack \overrightarrow{b},T_{\Omega,\alpha}]f\Vert_{L_{q_{1}}%
(B(x_{0},r))}\lesssim%
%TCIMACRO{\dprod \limits_{i=1}^{m}}%
%BeginExpansion
{\displaystyle \prod \limits_{i=1}^{m}}
%EndExpansion
\Vert \overrightarrow{b}\Vert_{LC_{p_{i},\lambda_{i}}^{\left \{  x_{0}\right \}
}}\,r^{\frac{n}{q_{1}}-\frac{n}{s}}%
%TCIMACRO{\dint \limits_{2r}^{\infty}}%
%BeginExpansion
{\displaystyle \int \limits_{2r}^{\infty}}
%EndExpansion
\left(  1+\ln \frac{t}{r}\right)  ^{m}t^{n\left(  -\frac{1}{q_{1}}+\left(
\frac{1}{s}+%
%TCIMACRO{\dsum \limits_{i=1}^{m}}%
%BeginExpansion
{\displaystyle \sum \limits_{i=1}^{m}}
%EndExpansion
\lambda_{i}+%
%TCIMACRO{\dsum \limits_{i=1}^{m}}%
%BeginExpansion
{\displaystyle \sum \limits_{i=1}^{m}}
%EndExpansion
\frac{1}{p_{i}}\right)  \right)  -1}\Vert f\Vert_{L_{p}(B(x_{0},t))}dt
\]
holds for any ball $B(x_{0},r)$ and for all $f\in L_{p}^{loc}({\mathbb{R}^{n}%
})$.
\end{lemma}

\begin{proof}
Without loss of generality, it is sufficient to show that the conclusion holds
for $[\overrightarrow{b},T_{\Omega,\alpha}]f=[\left(  b_{1},b_{2}\right)
,T_{\Omega,\alpha}]f$. As in the proof of Lemma \ref{lemma1}, we represent $f$
in form (\ref{e39}) and thus have%
\[
\left \Vert \lbrack \left(  b_{1},b_{2}\right)  ,T_{\Omega,\alpha}]f\right \Vert
_{L_{q_{1}}\left(  B\right)  }\leq \left \Vert \lbrack \left(  b_{1}%
,b_{2}\right)  ,T_{\Omega,\alpha}]f_{1}\right \Vert _{L_{q_{1}}\left(
B\right)  }+\left \Vert [\left(  b_{1},b_{2}\right)  ,T_{\Omega,\alpha}%
]f_{2}\right \Vert _{L_{q_{1}}\left(  B\right)  }=:F+G.
\]

Let us estimate $F+G$, respectively.

For $[\left(  b_{1},b_{2}\right)  ,T_{\Omega,\alpha}]f_{1}\left(  x\right)  $,
we have the following decomposition,%
\begin{align*}
\lbrack \left(  b_{1},b_{2}\right)  ,T_{\Omega,\alpha}]f_{1}\left(  x\right)
& =\left(  b_{1}\left(  x\right)  -\left(  b_{1}\right)  _{B}\right)  \left(
b_{2}\left(  x\right)  -\left(  b_{2}\right)  _{B}\right)  T_{\Omega,\alpha
}f_{1}\left(  x\right) \\
& -\left(  b_{1}\left(  \cdot \right)  -\left(  b_{1}\right)  _{B}\right)
T_{\Omega,\alpha}\left(  \left(  b_{2}\left(  \cdot \right)  -\left(
b_{2}\right)  _{B}\right)  f_{1}\right)  \left(  x\right) \\
& +\left(  b_{2}\left(  x\right)  -\left(  b_{2}\right)  _{B}\right)
T_{\Omega,\alpha}\left(  \left(  b_{1}\left(  x\right)  -\left(  b_{1}\right)
_{B}\right)  f_{1}\right)  \left(  x\right)  -\\
& T_{\Omega,\alpha}\left(  \left(  b_{1}\left(  \cdot \right)  -\left(
b_{1}\right)  _{B}\right)  \left(  b_{2}\left(  \cdot \right)  -\left(
b_{2}\right)  _{B}\right)  f_{1}\right)  \left(  x\right)  .
\end{align*}
Hence, we get%
\begin{align}
F  & =\left \Vert [\left(  b_{1},b_{2}\right)  ,T_{\Omega,\alpha}%
]f_{1}\right \Vert _{L_{q_{1}}\left(  B\right)  }\lesssim \nonumber \\
& \left \Vert \left(  b_{1}-\left(  b_{1}\right)  _{B}\right)  \left(
b_{2}\left(  x\right)  -\left(  b_{2}\right)  _{B}\right)  T_{\Omega,\alpha
}f_{1}\right \Vert _{L_{q_{1}}\left(  B\right)  }\nonumber \\
& +\left \Vert \left(  b_{1}-\left(  b_{1}\right)  _{B}\right)  T_{\Omega
,\alpha}\left(  \left(  b_{2}-\left(  b_{2}\right)  _{B}\right)  f_{1}\right)
\right \Vert _{L_{q_{1}}\left(  B\right)  }\nonumber \\
& +\left \Vert \left(  b_{2}-\left(  b_{2}\right)  _{B}\right)  T_{\Omega
,\alpha}\left(  \left(  b_{1}-\left(  b_{1}\right)  _{B}\right)  f_{1}\right)
\right \Vert _{L_{q_{1}}\left(  B\right)  }\nonumber \\
& +\left \Vert T_{\Omega,\alpha}\left(  \left(  b_{1}-\left(  b_{1}\right)
_{B}\right)  \left(  b_{2}-\left(  b_{2}\right)  _{B}\right)  f_{1}\right)
\right \Vert _{L_{q_{1}}\left(  B\right)  }\nonumber \\
& \equiv F_{1}+F_{2}+F_{3}+F_{4}.\label{0}%
\end{align}
One observes that the estimate of $F_{2}$ is analogous to that of $F_{3}$.
Thus, we will only estimate $F_{1}$, $F_{2}$ and $F_{4}$.

To estimate $F_{1}$, let $1<\overline{q},\overline{r}<\infty$, such
that$\frac{1}{\overline{q}}=\frac{1}{p}-\frac{\alpha}{n},$ $\frac{1}{q_{1}%
}=\frac{1}{\overline{r}}+\frac{1}{\overline{q}},\frac{1}{\overline{r}}%
=\frac{1}{p_{1}}+\frac{1}{p_{2}}$. Then, using H\"{o}lder's inequality and by
the boundedness of $T_{\Omega,\alpha}$ from $L_{p}$ into $L_{\overline{q}}$
(see Lemma \ref{Lemma1}) it follows that:%
\begin{align*}
F_{1}  & =\left \Vert \left(  b_{1}-\left(  b_{1}\right)  _{B}\right)  \left(
b_{2}\left(  x\right)  -\left(  b_{2}\right)  _{B}\right)  T_{\Omega,\alpha
}f_{1}\right \Vert _{L_{q_{1}}\left(  B\right)  }\\
& \lesssim \left \Vert \left(  b_{1}-\left(  b_{1}\right)  _{B}\right)  \left(
b_{2}-\left(  b_{2}\right)  _{B}\right)  \right \Vert _{L_{\overline{r}}\left(
B\right)  }\left \Vert T_{\Omega,\alpha}f_{1}\right \Vert _{L_{\overline{q}%
}\left(  B\right)  }\\
& \lesssim \left \Vert b_{1}-\left(  b_{1}\right)  _{B}\right \Vert _{L_{p_{1}%
}\left(  B\right)  }\left \Vert b_{2}-\left(  b_{2}\right)  _{B}\right \Vert
_{L_{p_{2}}\left(  B\right)  }\left \Vert f\right \Vert _{L_{p}\left(
2B\right)  }\\
& \lesssim \left \Vert b_{1}-\left(  b_{1}\right)  _{B}\right \Vert _{L_{p_{1}%
}\left(  B\right)  }\left \Vert b_{2}-\left(  b_{2}\right)  _{B}\right \Vert
_{L_{p_{2}}\left(  B\right)  }r^{n\left(  \frac{1}{p}-\frac{\alpha}{n}\right)
}\int \limits_{2r}^{\infty}\left \Vert f\right \Vert _{L_{q_{1}}\left(
B(x_{0},t)\right)  }\frac{dt}{t^{\frac{n}{p}+1-\alpha}}.
\end{align*}
From Lemma \ref{Lemma 4}, it is easy to see that%
\begin{equation}
\left \Vert b_{i}-\left(  b_{i}\right)  _{B}\right \Vert _{L_{p_{i}}\left(
B\right)  }\leq Cr^{\frac{n}{p_{i}}+n\lambda_{i}}\left \Vert b_{i}\right \Vert
_{LC_{p_{i},\lambda_{i}}^{\left \{  x_{0}\right \}  }},\label{1}%
\end{equation}
and%
\begin{align}
\left \Vert b_{i}-\left(  b_{i}\right)  _{B}\right \Vert _{L_{p_{i}}\left(
2B\right)  } &  \leq \left \Vert b_{i}-\left(  b_{i}\right)  _{2B}\right \Vert
_{L_{p_{i}}\left(  2B\right)  }+\left \Vert \left(  b_{i}\right)  _{B}-\left(
b_{i}\right)  _{2B}\right \Vert _{L_{p_{i}}\left(  2B\right)  }\nonumber \\
&  \lesssim r^{\frac{n}{p_{i}}+n\lambda_{i}}\left \Vert b_{i}\right \Vert
_{LC_{p_{i},\lambda_{i}}^{\left \{  x_{0}\right \}  }},\label{2*}%
\end{align}
for $i=1$, $2$. Hence, by (\ref{1}) we get%
\begin{align*}
F_{1}  & \lesssim \Vert b_{1}\Vert_{LC_{p_{1},\lambda_{1}}^{\left \{
x_{0}\right \}  }}\Vert b_{2}\Vert_{LC_{p_{2},\lambda_{2}}^{\left \{
x_{0}\right \}  }}r^{n\left(  \frac{1}{p_{1}}+\frac{1}{p_{2}}+\frac{1}{p}%
-\frac{\alpha}{n}\right)  }\int \limits_{2r}^{\infty}\left(  1+\ln \frac{t}%
{r}\right)  ^{2}t^{-\frac{n}{p}+n\left(  \lambda_{1}+\lambda_{2}\right)
-1+\alpha}\left \Vert f\right \Vert _{L_{p}\left(  B(x_{0},t)\right)  }dt\\
& \lesssim \Vert b_{1}\Vert_{LC_{p_{1},\lambda_{1}}^{\left \{  x_{0}\right \}  }%
}\Vert b_{2}\Vert_{LC_{p_{2},\lambda_{2}}^{\left \{  x_{0}\right \}  }}%
r^{\frac{n}{q_{1}}}\int \limits_{2r}^{\infty}\left(  1+\ln \frac{t}{r}\right)
^{2}t^{-\frac{n}{q_{1}}+n\left(  \lambda_{1}+\lambda_{2}\right)  +n\left(
\frac{1}{p_{1}}+\frac{1}{p_{2}}\right)  -1}\left \Vert f\right \Vert
_{L_{p}\left(  B(x_{0},t)\right)  }dt.
\end{align*}
To estimate $F_{2}$, let $1<\tau<\infty$, such that $\frac{1}{q_{1}}=\frac
{1}{p_{1}}+\frac{1}{\tau}$. Then, similar to the estimates for $F_{1}$, we
have%
\begin{align*}
F_{2} &  =\left \Vert \left(  b_{1}-\left(  b_{1}\right)  _{B}\right)
T_{\Omega,\alpha}\left(  \left(  b_{2}-\left(  b_{2}\right)  _{B}\right)
f_{1}\right)  \right \Vert _{L_{q_{1}}\left(  B\right)  }\\
&  \lesssim \left \Vert b_{1}-\left(  b_{1}\right)  _{B}\right \Vert _{L_{p_{1}%
}\left(  B\right)  }\left \Vert T_{\Omega,\alpha}\left(  \left(  b_{2}\left(
\cdot \right)  -\left(  b_{2}\right)  _{B}\right)  f_{1}\right)  \right \Vert
_{L_{\tau}\left(  B\right)  }\\
&  \lesssim \left \Vert b_{1}-\left(  b_{1}\right)  _{B}\right \Vert _{L_{p_{1}%
}\left(  B\right)  }\left \Vert \left(  b_{2}\left(  \cdot \right)  -\left(
b_{2}\right)  _{B}\right)  f_{1}\right \Vert _{L_{k}\left(  B\right)  }\\
&  \lesssim \left \Vert b_{1}-\left(  b_{1}\right)  _{B}\right \Vert _{L_{p_{1}%
}\left(  B\right)  }\left \Vert b_{2}-\left(  b_{2}\right)  _{B}\right \Vert
_{L_{p_{2}}\left(  2B\right)  }\left \Vert f\right \Vert _{_{L_{p}\left(
2B\right)  }},
\end{align*}
where $1<k<\frac{2n}{\alpha}$, such that $\frac{1}{k}=\frac{1}{p_{2}}+\frac
{1}{p}=\frac{1}{\tau}+\frac{\alpha}{n}$. By (\ref{1}) and (\ref{2*}), we get%
\[
F_{2}\lesssim \Vert b_{1}\Vert_{LC_{p_{1},\lambda_{1}}^{\left \{  x_{0}\right \}
}}\Vert b_{2}\Vert_{LC_{p_{2},\lambda_{2}}^{\left \{  x_{0}\right \}  }}%
r^{\frac{n}{q_{1}}}\int \limits_{2r}^{\infty}\left(  1+\ln \frac{t}{r}\right)
^{2}t^{-\frac{n}{q_{1}}+n\left(  \lambda_{1}+\lambda_{2}\right)  +n\left(
\frac{1}{p_{1}}+\frac{1}{p_{2}}\right)  -1}\left \Vert f\right \Vert
_{L_{p}\left(  B(x_{0},t)\right)  }dt.
\]
In a similar way, $F_{3}$ has the same estimate as above, so we omit the
details. Then we have that%
\[
F_{3}\lesssim \Vert b_{1}\Vert_{LC_{p_{1},\lambda_{1}}^{\left \{  x_{0}\right \}
}}\Vert b_{2}\Vert_{LC_{p_{2},\lambda_{2}}^{\left \{  x_{0}\right \}  }}%
r^{\frac{n}{q_{1}}}\int \limits_{2r}^{\infty}\left(  1+\ln \frac{t}{r}\right)
^{2}t^{-\frac{n}{q_{1}}+n\left(  \lambda_{1}+\lambda_{2}\right)  +n\left(
\frac{1}{p_{1}}+\frac{1}{p_{2}}\right)  -1}\left \Vert f\right \Vert
_{L_{p}\left(  B(x_{0},t)\right)  }dt.
\]

Now let us consider the term $F_{4}$. Let $1<q,\tau<\frac{2n}{\alpha}$, such
that $\frac{1}{q}=\frac{1}{\tau}+\frac{1}{p}$, $\frac{1}{\tau}=\frac{1}{p_{1}%
}+\frac{1}{p_{2}}$ and $\frac{1}{q_{1}}=\frac{1}{q}-\frac{\alpha}{n}$. Then by
the boundedness of $T_{\Omega,\alpha}$ from $L_{q}$ into $L_{q_{1}}$ (see
Lemma \ref{Lemma1}), H\"{o}lder's inequality and (\ref{2*}), we obtain%
\begin{align*}
F_{4}  & =\left \Vert T_{\Omega,\alpha}\left(  \left(  b_{1}-\left(
b_{1}\right)  _{B}\right)  \left(  b_{2}-\left(  b_{2}\right)  _{B}\right)
f_{1}\right)  \right \Vert _{L_{q_{1}}\left(  B\right)  }\\
& \lesssim \left \Vert \left(  b_{1}-\left(  b_{1}\right)  _{B}\right)  \left(
b_{2}-\left(  b_{2}\right)  _{B}\right)  f_{1}\right \Vert _{L_{q}\left(
B\right)  }\\
& \lesssim \left \Vert \left(  b_{1}-\left(  b_{1}\right)  _{B}\right)  \left(
b_{2}-\left(  b_{2}\right)  _{B}\right)  \right \Vert _{L_{\tau}\left(
B\right)  }\left \Vert f_{1}\right \Vert _{L_{p}\left(  B\right)  }\\
& \lesssim \left \Vert b_{1}-\left(  b_{1}\right)  _{B}\right \Vert _{L_{p_{1}%
}\left(  2B\right)  }\left \Vert b_{2}-\left(  b_{2}\right)  _{B}\right \Vert
_{L_{p_{2}}\left(  2B\right)  }\left \Vert f\right \Vert _{_{L_{p}\left(
2B\right)  }}\\
& \lesssim \Vert b_{1}\Vert_{LC_{p_{1},\lambda_{1}}^{\left \{  x_{0}\right \}  }%
}\Vert b_{2}\Vert_{LC_{p_{2},\lambda_{2}}^{\left \{  x_{0}\right \}  }}%
r^{\frac{n}{q_{1}}}\int \limits_{2r}^{\infty}\left(  1+\ln \frac{t}{r}\right)
^{2}t^{-\frac{n}{q_{1}}+n\left(  \lambda_{1}+\lambda_{2}\right)  +n\left(
\frac{1}{p_{1}}+\frac{1}{p_{2}}\right)  -1}\left \Vert f\right \Vert
_{L_{p}\left(  B(x_{0},t)\right)  }dt.
\end{align*}
Combining all the estimates of $F_{1}$, $F_{2}$, $F_{3}$, $F_{4}$; we get%
\begin{align*}
F  & =\left \Vert [\left(  b_{1},b_{2}\right)  ,T_{\Omega,\alpha}%
]f_{1}\right \Vert _{L_{q_{1}}\left(  B\right)  }\lesssim \Vert b_{1}%
\Vert_{LC_{p_{1},\lambda_{1}}^{\left \{  x_{0}\right \}  }}\Vert b_{2}%
\Vert_{LC_{p_{2},\lambda_{2}}^{\left \{  x_{0}\right \}  }}r^{\frac{n}{q_{1}}}\\
& \times \int \limits_{2r}^{\infty}\left(  1+\ln \frac{t}{r}\right)
^{2}t^{-\frac{n}{q_{1}}+n\left(  \lambda_{1}+\lambda_{2}\right)  +n\left(
\frac{1}{p_{1}}+\frac{1}{p_{2}}\right)  -1}\left \Vert f\right \Vert
_{L_{p}\left(  B(x_{0},t)\right)  }dt.
\end{align*}

Now, let us estimate $G=\left \Vert [\left(  b_{1},b_{2}\right)  ,T_{\Omega
,\alpha}]f_{2}\right \Vert _{L_{q_{1}}\left(  B\right)  }$. For $G$, it's
similar to (\ref{0}) we also write%
\begin{align*}
G  & =\left \Vert [\left(  b_{1},b_{2}\right)  ,T_{\Omega,\alpha}%
]f_{2}\right \Vert _{L_{q_{1}}\left(  B\right)  }\lesssim \\
& \left \Vert \left(  b_{1}-\left(  b_{1}\right)  _{B}\right)  \left(
b_{2}\left(  x\right)  -\left(  b_{2}\right)  _{B}\right)  T_{\Omega,\alpha
}f_{2}\right \Vert _{L_{q_{1}}\left(  B\right)  }\\
& +\left \Vert \left(  b_{1}-\left(  b_{1}\right)  _{B}\right)  T_{\Omega
,\alpha}\left(  \left(  b_{2}-\left(  b_{2}\right)  _{B}\right)  f_{2}\right)
\right \Vert _{L_{q_{1}}\left(  B\right)  }\\
& +\left \Vert \left(  b_{2}-\left(  b_{2}\right)  _{B}\right)  T_{\Omega
,\alpha}\left(  \left(  b_{1}-\left(  b_{1}\right)  _{B}\right)  f_{2}\right)
\right \Vert _{L_{q_{1}}\left(  B\right)  }\\
& +\left \Vert T_{\Omega,\alpha}\left(  \left(  b_{1}-\left(  b_{1}\right)
_{B}\right)  \left(  b_{2}-\left(  b_{2}\right)  _{B}\right)  f_{2}\right)
\right \Vert _{L_{q_{1}}\left(  B\right)  }\\
& \equiv G_{1}+G_{2}+G_{3}+G_{4}.
\end{align*}
To estimate $G_{1}$, let $1<p_{1},p_{2}<\frac{2n}{\alpha}$, such that
$\frac{1}{q_{1}}=\frac{1}{\overline{p}}+\frac{1}{\overline{q}}$, $\frac
{1}{\overline{p}}=\frac{1}{p_{1}}+\frac{1}{p_{2}}$ and $\frac{1}{\overline{q}%
}=\frac{1}{p}-\frac{\alpha}{n}$. Then, using H\"{o}lder's inequality, noting
that in (\ref{10}) and by (\ref{1}), we have%
\begin{align*}
G_{1}  & =\left \Vert \left(  b_{1}-\left(  b_{1}\right)  _{B}\right)  \left(
b_{2}\left(  x\right)  -\left(  b_{2}\right)  _{B}\right)  T_{\Omega,\alpha
}f_{2}\right \Vert _{L_{q_{1}}\left(  B\right)  }\\
& \lesssim \left \Vert \left(  b_{1}-\left(  b_{1}\right)  _{B}\right)  \left(
b_{2}-\left(  b_{2}\right)  _{B}\right)  \right \Vert _{L_{\overline{p}}\left(
B\right)  }\left \Vert T_{\Omega,\alpha}f_{2}\right \Vert _{L_{\overline{q}%
}\left(  B\right)  }\\
& \lesssim \left \Vert b_{1}-\left(  b_{1}\right)  _{B}\right \Vert _{L_{p_{1}%
}\left(  B\right)  }\left \Vert b_{2}-\left(  b_{2}\right)  _{B}\right \Vert
_{L_{p_{2}}\left(  B\right)  }r^{\frac{n}{_{\overline{q}}}}\int \limits_{2r}%
^{\infty}\left \Vert f\right \Vert _{L_{p}\left(  B\left(  x_{0},t\right)
\right)  }t^{-\frac{n}{\overline{q}}-1}dt\\
& \lesssim \Vert b_{1}\Vert_{LC_{p_{1},\lambda_{1}}^{\left \{  x_{0}\right \}  }%
}\Vert b_{2}\Vert_{LC_{p_{2},\lambda_{2}}^{\left \{  x_{0}\right \}  }%
}r^{n\left(  \frac{1}{p_{1}}+\frac{1}{p_{2}}+\frac{1}{p}-\frac{\alpha}%
{n}\right)  }\\
& \times \int \limits_{2r}^{\infty}\left(  1+\ln \frac{t}{r}\right)
^{2}t^{-\frac{n}{p}+n\left(  \lambda_{1}+\lambda_{2}\right)  -1+\alpha
}\left \Vert f\right \Vert _{L_{p}\left(  B(x_{0},t)\right)  }dt\\
& \lesssim \Vert b_{1}\Vert_{LC_{p_{1},\lambda_{1}}^{\left \{  x_{0}\right \}  }%
}\Vert b_{2}\Vert_{LC_{p_{2},\lambda_{2}}^{\left \{  x_{0}\right \}  }}%
r^{\frac{n}{q_{1}}}\int \limits_{2r}^{\infty}\left(  1+\ln \frac{t}{r}\right)
^{2}t^{-\frac{n}{q_{1}}+n\left(  \lambda_{1}+\lambda_{2}\right)  +n\left(
\frac{1}{p_{1}}+\frac{1}{p_{2}}\right)  -1}\left \Vert f\right \Vert
_{L_{p}\left(  B(x_{0},t)\right)  }dt.
\end{align*}

On the other hand, for the estimates used in $G_{2}$, $G_{3}$, we have to
prove the below inequality:%
\begin{equation}
\left \vert T_{\Omega,\alpha}\left(  \left(  b_{2}-\left(  b_{2}\right)
_{B}\right)  f_{2}\right)  \left(  x\right)  \right \vert \lesssim \left \Vert
b\right \Vert _{LC_{p_{2},\lambda}^{\left \{  x_{0}\right \}  }}%
%TCIMACRO{\dint \limits_{2r}^{\infty}}%
%BeginExpansion
{\displaystyle \int \limits_{2r}^{\infty}}
%EndExpansion
\left(  1+\ln \frac{t}{r}\right)  t^{-\frac{n}{p}+n\lambda_{2}-1+\alpha
}\left \Vert f\right \Vert _{L_{p}\left(  B\left(  x_{0},t\right)  \right)
}dt.\label{11}%
\end{equation}
Indeed, when $s^{\prime}\leq q$, for $x\in B$, by Fubini's theorem and
applying H\"{o}lder's inequality and from (\ref{b}), (\ref{c}), (\ref{e311})
we have

$\left \vert T_{\Omega,\alpha}\left(  \left(  b_{2}-\left(  b_{2}\right)
_{B}\right)  f_{2}\right)  \left(  x\right)  \right \vert \lesssim%
%TCIMACRO{\dint \limits_{\left(  2B\right)  ^{C}}}%
%BeginExpansion
{\displaystyle \int \limits_{\left(  2B\right)  ^{C}}}
%EndExpansion
\left \vert b_{2}\left(  y\right)  -\left(  b_{2}\right)  _{B}\right \vert
\left \vert \Omega(x,x-y)\right \vert \frac{\left \vert f\left(  y\right)
\right \vert }{\left \vert x-y\right \vert ^{n-\alpha}}dy$

$\lesssim%
%TCIMACRO{\dint \limits_{\left(  2B\right)  ^{C}}}%
%BeginExpansion
{\displaystyle \int \limits_{\left(  2B\right)  ^{C}}}
%EndExpansion
\left \vert b_{2}\left(  y\right)  -\left(  b_{2}\right)  _{B}\right \vert
\left \vert \Omega(x,x-y)\right \vert \frac{\left \vert f\left(  y\right)
\right \vert }{\left \vert x_{0}-y\right \vert ^{n-\alpha}}dy$

$\approx%
%TCIMACRO{\dint \limits_{2r}^{\infty}}%
%BeginExpansion
{\displaystyle \int \limits_{2r}^{\infty}}
%EndExpansion%
%TCIMACRO{\dint \limits_{2r<\left\vert x_{0}-y\right\vert <t}}%
%BeginExpansion
{\displaystyle \int \limits_{2r<\left \vert x_{0}-y\right \vert <t}}
%EndExpansion
\left \vert b_{2}\left(  y\right)  -\left(  b_{2}\right)  _{B}\right \vert
\left \vert \Omega(x,x-y)\right \vert \left \vert f\left(  y\right)  \right \vert
dy\frac{dt}{t^{n-\alpha+1}}$

$\lesssim%
%TCIMACRO{\dint \limits_{2r}^{\infty}}%
%BeginExpansion
{\displaystyle \int \limits_{2r}^{\infty}}
%EndExpansion%
%TCIMACRO{\dint \limits_{B\left(  x_{0},t\right)  }}%
%BeginExpansion
{\displaystyle \int \limits_{B\left(  x_{0},t\right)  }}
%EndExpansion
\left \vert b_{2}\left(  y\right)  -\left(  b_{2}\right)  _{B\left(
x_{0},t\right)  }\right \vert \left \vert \Omega(x,x-y)\right \vert \left \vert
f\left(  y\right)  \right \vert dy\frac{dt}{t^{n-\alpha+1}}$

$+%
%TCIMACRO{\dint \limits_{2r}^{\infty}}%
%BeginExpansion
{\displaystyle \int \limits_{2r}^{\infty}}
%EndExpansion
\left \vert \left(  b_{2}\right)  _{B\left(  x_{0},r\right)  }-\left(
b_{2}\right)  _{B\left(  x_{0},t\right)  }\right \vert
%TCIMACRO{\dint \limits_{B\left(  x_{0},t\right)  }}%
%BeginExpansion
{\displaystyle \int \limits_{B\left(  x_{0},t\right)  }}
%EndExpansion
\left \vert \Omega(x,x-y)\right \vert \left \vert f\left(  y\right)  \right \vert
dy\frac{dt}{t^{n-\alpha+1}}$

$\lesssim%
%TCIMACRO{\dint \limits_{2r}^{\infty}}%
%BeginExpansion
{\displaystyle \int \limits_{2r}^{\infty}}
%EndExpansion
\left \Vert b_{2}\left(  \cdot \right)  -\left(  b_{2}\right)  _{B\left(
x_{0},t\right)  }\right \Vert _{L_{p_{2}}\left(  B\left(  x_{0},t\right)
\right)  }\left \Vert \Omega(x,x-\cdot)\right \Vert _{L_{s}\left(  B\left(
x_{0},t\right)  \right)  }\left \Vert f\right \Vert _{L_{p}\left(
B(x_{0},t)\right)  }\left \vert B\left(  x_{0},t\right)  \right \vert
^{1-\frac{1}{p_{2}}-\frac{1}{s}-\frac{1}{p}}\frac{dt}{t^{n-\alpha+1}}$

$+%
%TCIMACRO{\dint \limits_{2r}^{\infty}}%
%BeginExpansion
{\displaystyle \int \limits_{2r}^{\infty}}
%EndExpansion
\left \vert \left(  b_{2}\right)  _{B\left(  x_{0},r\right)  }-\left(
b_{2}\right)  _{B\left(  x_{0},t\right)  }\right \vert \left \Vert f\right \Vert
_{L_{p}\left(  B\left(  x_{0},t\right)  \right)  }\left \Vert \Omega
(x,x-\cdot)\right \Vert _{L_{s}\left(  B\left(  x_{0},t\right)  \right)
}\left \vert B\left(  x_{0},t\right)  \right \vert ^{1-\frac{1}{p}-\frac{1}{s}%
}\frac{dt}{t^{n-\alpha+1}}$

$\lesssim%
%TCIMACRO{\dint \limits_{2r}^{\infty}}%
%BeginExpansion
{\displaystyle \int \limits_{2r}^{\infty}}
%EndExpansion
\left \Vert b_{2}\left(  \cdot \right)  -\left(  b_{2}\right)  _{B\left(
x_{0},t\right)  }\right \Vert _{L_{p_{2}}\left(  B\left(  x_{0},t\right)
\right)  }\left \Vert f\right \Vert _{L_{p}\left(  B\left(  x_{0},t\right)
\right)  }\frac{dt}{t^{n\left(  \frac{1}{p_{2}}+\frac{1}{p}\right)  -\alpha
+1}}$

$+\left \Vert b\right \Vert _{LC_{p_{2},\lambda}^{\left \{  x_{0}\right \}  }}%
%TCIMACRO{\dint \limits_{2r}^{\infty}}%
%BeginExpansion
{\displaystyle \int \limits_{2r}^{\infty}}
%EndExpansion
\left(  1+\ln \frac{t}{r}\right)  t^{-\frac{n}{p}+n\lambda_{2}-1+\alpha
}\left \Vert f\right \Vert _{L_{p}\left(  B\left(  x_{0},t\right)  \right)  }dt$

$\lesssim \left \Vert b\right \Vert _{LC_{p_{2},\lambda}^{\left \{  x_{0}\right \}
}}%
%TCIMACRO{\dint \limits_{2r}^{\infty}}%
%BeginExpansion
{\displaystyle \int \limits_{2r}^{\infty}}
%EndExpansion
\left(  1+\ln \frac{t}{r}\right)  t^{-\frac{n}{p}+n\lambda_{2}-1+\alpha
}\left \Vert f\right \Vert _{L_{p}\left(  B\left(  x_{0},t\right)  \right)
}dt.$

This completes the proof of inequality (\ref{11}).

Let $1<\tau<\infty$, such that $\frac{1}{q_{1}}=\frac{1}{p_{1}}+\frac{1}{\tau
}$ and $\frac{1}{\tau}=\frac{1}{p_{2}}+\frac{1}{p}-\frac{\alpha}{n}$. Then,
using H\"{o}lder's inequality and from (\ref{11}) and (\ref{c}), we get%
\begin{align*}
G_{2} &  =\left \Vert \left(  b_{1}-\left(  b_{1}\right)  _{B}\right)
T_{\Omega,\alpha}\left(  \left(  b_{2}-\left(  b_{2}\right)  _{B}\right)
f_{2}\right)  \right \Vert _{L_{q_{1}}\left(  B\right)  }\\
&  \lesssim \left \Vert b_{1}-\left(  b_{1}\right)  _{B}\right \Vert _{L_{p_{1}%
}\left(  B\right)  }\left \Vert T_{\Omega,\alpha}\left(  \left(  b_{2}\left(
\cdot \right)  -\left(  b_{2}\right)  _{B}\right)  f_{2}\right)  \right \Vert
_{L_{\tau}\left(  B\right)  }\\
&  \lesssim \left \Vert b_{1}-\left(  b_{1}\right)  _{B}\right \Vert _{L_{p_{1}%
}\left(  B\right)  }\left \Vert b_{2}\right \Vert _{LC_{p_{2},\lambda_{2}%
}^{\left \{  x_{0}\right \}  }}r^{\frac{n}{\tau}}%
%TCIMACRO{\dint \limits_{2r}^{\infty}}%
%BeginExpansion
{\displaystyle \int \limits_{2r}^{\infty}}
%EndExpansion
\left(  1+\ln \frac{t}{r}\right)  t^{-\frac{n}{p}+n\lambda_{2}-1+\alpha
}\left \Vert f\right \Vert _{L_{p}\left(  B\left(  x_{0},t\right)  \right)
}dt\\
&  \lesssim \Vert b_{1}\Vert_{LC_{p_{1},\lambda_{1}}^{\left \{  x_{0}\right \}
}}\Vert b_{2}\Vert_{LC_{p_{2},\lambda_{2}}^{\left \{  x_{0}\right \}  }}%
r^{\frac{n}{q_{1}}}\int \limits_{2r}^{\infty}\left(  1+\ln \frac{t}{r}\right)
^{2}t^{-\frac{n}{q_{1}}+n\left(  \lambda_{1}+\lambda_{2}\right)  +n\left(
\frac{1}{p_{1}}+\frac{1}{p_{2}}\right)  -1}\left \Vert f\right \Vert
_{L_{p}\left(  B(x_{0},t)\right)  }dt.
\end{align*}
Similarly, $G_{3}$ has the same estimate above, so here we omit the details.
Then the inequality%
\begin{align*}
G_{3}  & =\left \Vert \left(  b_{2}-\left(  b_{2}\right)  _{B}\right)
T_{\Omega,\alpha}\left(  \left(  b_{1}-\left(  b_{1}\right)  _{B}\right)
f_{2}\right)  \right \Vert _{L_{q_{1}}\left(  B\right)  }\\
& \lesssim \Vert b_{1}\Vert_{LC_{p_{1},\lambda_{1}}^{\left \{  x_{0}\right \}  }%
}\Vert b_{2}\Vert_{LC_{p_{2},\lambda_{2}}^{\left \{  x_{0}\right \}  }}%
r^{\frac{n}{q_{1}}}\int \limits_{2r}^{\infty}\left(  1+\ln \frac{t}{r}\right)
^{2}t^{-\frac{n}{q_{1}}+n\left(  \lambda_{1}+\lambda_{2}\right)  +n\left(
\frac{1}{p_{1}}+\frac{1}{p_{2}}\right)  -1}\left \Vert f\right \Vert
_{L_{p}\left(  B(x_{0},t)\right)  }dt
\end{align*}
is valid.

Now, let us estimate $G_{4}=\left \Vert T_{\Omega,\alpha}\left(  \left(
b_{1}-\left(  b_{1}\right)  _{B}\right)  \left(  b_{2}-\left(  b_{2}\right)
_{B}\right)  f_{2}\right)  \right \Vert _{L_{q_{1}}\left(  B\right)  }$. It's
similar to the estimate of (\ref{11}), for any $x\in B$, we also write

$\left \vert T_{\Omega,\alpha}\left(  \left(  b_{1}-\left(  b_{1}\right)
_{B}\right)  \left(  b_{2}-\left(  b_{2}\right)  _{B}\right)  f_{2}\right)
\left(  x\right)  \right \vert $

$\lesssim%
%TCIMACRO{\dint \limits_{2r}^{\infty}}%
%BeginExpansion
{\displaystyle \int \limits_{2r}^{\infty}}
%EndExpansion%
%TCIMACRO{\dint \limits_{B\left(  x_{0},t\right)  }}%
%BeginExpansion
{\displaystyle \int \limits_{B\left(  x_{0},t\right)  }}
%EndExpansion
\left \vert b_{1}\left(  y\right)  -\left(  b_{1}\right)  _{B\left(
x_{0},t\right)  }\right \vert \left \vert b_{2}\left(  y\right)  -\left(
b_{2}\right)  _{B\left(  x_{0},t\right)  }\right \vert \left \vert
\Omega(x,x-y)\right \vert \left \vert f\left(  y\right)  \right \vert dy\frac
{dt}{t^{n-\alpha+1}}$

$+%
%TCIMACRO{\dint \limits_{2r}^{\infty}}%
%BeginExpansion
{\displaystyle \int \limits_{2r}^{\infty}}
%EndExpansion%
%TCIMACRO{\dint \limits_{B\left(  x_{0},t\right)  }}%
%BeginExpansion
{\displaystyle \int \limits_{B\left(  x_{0},t\right)  }}
%EndExpansion
\left \vert b_{1}\left(  y\right)  -\left(  b_{1}\right)  _{B\left(
x_{0},t\right)  }\right \vert \left \vert \left(  b_{2}\right)  _{B\left(
x_{0},t\right)  }-\left(  b_{2}\right)  _{B\left(  x_{0},r\right)
}\right \vert \left \vert \Omega(x,x-y)\right \vert \left \vert f\left(  y\right)
\right \vert dy\frac{dt}{t^{n-\alpha+1}}$

$+%
%TCIMACRO{\dint \limits_{2r}^{\infty}}%
%BeginExpansion
{\displaystyle \int \limits_{2r}^{\infty}}
%EndExpansion%
%TCIMACRO{\dint \limits_{B\left(  x_{0},t\right)  }}%
%BeginExpansion
{\displaystyle \int \limits_{B\left(  x_{0},t\right)  }}
%EndExpansion
\left \vert \left(  b_{1}\right)  _{B\left(  x_{0},t\right)  }-\left(
b_{2}\right)  _{B\left(  x_{0},r\right)  }\right \vert \left \vert b_{2}\left(
y\right)  -\left(  b_{2}\right)  _{B\left(  x_{0},t\right)  }\right \vert
\left \vert \Omega(x,x-y)\right \vert \left \vert f\left(  y\right)  \right \vert
dy\frac{dt}{t^{n-\alpha+1}}$

$+%
%TCIMACRO{\dint \limits_{2r}^{\infty}}%
%BeginExpansion
{\displaystyle \int \limits_{2r}^{\infty}}
%EndExpansion%
%TCIMACRO{\dint \limits_{B\left(  x_{0},t\right)  }}%
%BeginExpansion
{\displaystyle \int \limits_{B\left(  x_{0},t\right)  }}
%EndExpansion
\left \vert \left(  b_{1}\right)  _{B\left(  x_{0},t\right)  }-\left(
b_{2}\right)  _{B\left(  x_{0},r\right)  }\right \vert \left \vert \left(
b_{2}\right)  _{B\left(  x_{0},t\right)  }-\left(  b_{2}\right)  _{B\left(
x_{0},r\right)  }\right \vert \left \vert \Omega(x,x-y)\right \vert \left \vert
f\left(  y\right)  \right \vert dy\frac{dt}{t^{n-\alpha+1}}$

$\equiv G_{41}+G_{42}+G_{43}+G_{44}.$

Let us estimate $G_{41}$, $G_{42}$, $G_{43}$, $G_{44}$, respectively.

Firstly, to estimate $G_{41}$, similar to the estimate of (\ref{11}), we get%
\[
G_{41}\lesssim \Vert b_{1}\Vert_{LC_{p_{1},\lambda_{1}}^{\left \{
x_{0}\right \}  }}\Vert b_{2}\Vert_{LC_{p_{2},\lambda_{2}}^{\left \{
x_{0}\right \}  }}\int \limits_{2r}^{\infty}\left(  1+\ln \frac{t}{r}\right)
^{2}t^{-\frac{n}{q_{1}}+n\left(  \lambda_{1}+\lambda_{2}\right)  +n\left(
\frac{1}{p_{1}}+\frac{1}{p_{2}}\right)  -1}\left \Vert f\right \Vert
_{L_{p}\left(  B(x_{0},t)\right)  }dt.
\]
Secondly, to estimate $G_{42}$ and $G_{43}$, from (\ref{11}), (\ref{b}) and
(\ref{c}), it follows that%
\[
G_{42}\lesssim \Vert b_{1}\Vert_{LC_{p_{1},\lambda_{1}}^{\left \{
x_{0}\right \}  }}\Vert b_{2}\Vert_{LC_{p_{2},\lambda_{2}}^{\left \{
x_{0}\right \}  }}\int \limits_{2r}^{\infty}\left(  1+\ln \frac{t}{r}\right)
^{2}t^{-\frac{n}{q_{1}}+n\left(  \lambda_{1}+\lambda_{2}\right)  +n\left(
\frac{1}{p_{1}}+\frac{1}{p_{2}}\right)  -1}\left \Vert f\right \Vert
_{L_{p}\left(  B(x_{0},t)\right)  }dt,
\]
and%
\[
G_{43}\lesssim \Vert b_{1}\Vert_{LC_{p_{1},\lambda_{1}}^{\left \{
x_{0}\right \}  }}\Vert b_{2}\Vert_{LC_{p_{2},\lambda_{2}}^{\left \{
x_{0}\right \}  }}\int \limits_{2r}^{\infty}\left(  1+\ln \frac{t}{r}\right)
^{2}t^{-\frac{n}{q_{1}}+n\left(  \lambda_{1}+\lambda_{2}\right)  +n\left(
\frac{1}{p_{1}}+\frac{1}{p_{2}}\right)  -1}\left \Vert f\right \Vert
_{L_{p}\left(  B(x_{0},t)\right)  }dt.
\]
Finally, to estimate $G_{44}$, similar to the estimate of (\ref{11}) and from
(\ref{b}) and (\ref{c}), we have%
\[
G_{44}\lesssim \Vert b_{1}\Vert_{LC_{p_{1},\lambda_{1}}^{\left \{
x_{0}\right \}  }}\Vert b_{2}\Vert_{LC_{p_{2},\lambda_{2}}^{\left \{
x_{0}\right \}  }}\int \limits_{2r}^{\infty}\left(  1+\ln \frac{t}{r}\right)
^{2}t^{-\frac{n}{q_{1}}+n\left(  \lambda_{1}+\lambda_{2}\right)  +n\left(
\frac{1}{p_{1}}+\frac{1}{p_{2}}\right)  -1}\left \Vert f\right \Vert
_{L_{p}\left(  B(x_{0},t)\right)  }dt.
\]
By the estimates of $G_{4j}$ above, where $j=1$, $2$, $3$, we know that
\begin{align*}
\left \vert T_{\Omega,\alpha}\left(  \left(  b_{1}-\left(  b_{1}\right)
_{B}\right)  \left(  b_{2}-\left(  b_{2}\right)  _{B}\right)  f_{2}\right)
\left(  x\right)  \right \vert  & \lesssim \Vert b_{1}\Vert_{LC_{p_{1}%
,\lambda_{1}}^{\left \{  x_{0}\right \}  }}\Vert b_{2}\Vert_{LC_{p_{2}%
,\lambda_{2}}^{\left \{  x_{0}\right \}  }}\\
& \times \int \limits_{2r}^{\infty}\left(  1+\ln \frac{t}{r}\right)
^{2}t^{-\frac{n}{q_{1}}+n\left(  \lambda_{1}+\lambda_{2}\right)  +n\left(
\frac{1}{p_{1}}+\frac{1}{p_{2}}\right)  -1}\left \Vert f\right \Vert
_{L_{p}\left(  B(x_{0},t)\right)  }dt.
\end{align*}
Then, we have%
\begin{align*}
G_{4}  & =\left \Vert T_{\Omega,\alpha}\left(  \left(  b_{1}-\left(
b_{1}\right)  _{B}\right)  \left(  b_{2}-\left(  b_{2}\right)  _{B}\right)
f_{2}\right)  \right \Vert _{L_{q_{1}}\left(  B\right)  }\lesssim \Vert
b_{1}\Vert_{LC_{p_{1},\lambda_{1}}^{\left \{  x_{0}\right \}  }}\Vert b_{2}%
\Vert_{LC_{p_{2},\lambda_{2}}^{\left \{  x_{0}\right \}  }}r^{\frac{n}{q_{1}}}\\
& \times \int \limits_{2r}^{\infty}\left(  1+\ln \frac{t}{r}\right)
^{2}t^{-\frac{n}{q_{1}}+n\left(  \lambda_{1}+\lambda_{2}\right)  +n\left(
\frac{1}{p_{1}}+\frac{1}{p_{2}}\right)  -1}\left \Vert f\right \Vert
_{L_{p}\left(  B(x_{0},t)\right)  }dt.
\end{align*}
So, combining all the estimates for $G_{1},$ $G_{2}$, $G_{3}$, $G_{4}$, we get%
\begin{align*}
G  & =\left \Vert [\left(  b_{1},b_{2}\right)  ,T_{\Omega,\alpha}%
]f_{2}\right \Vert _{L_{q_{1}}\left(  B\right)  }\lesssim \Vert b_{1}%
\Vert_{LC_{p_{1},\lambda_{1}}^{\left \{  x_{0}\right \}  }}\Vert b_{2}%
\Vert_{LC_{p_{2},\lambda_{2}}^{\left \{  x_{0}\right \}  }}r^{\frac{n}{q_{1}}}\\
& \times \int \limits_{2r}^{\infty}\left(  1+\ln \frac{t}{r}\right)
^{2}t^{-\frac{n}{q_{1}}+n\left(  \lambda_{1}+\lambda_{2}\right)  +n\left(
\frac{1}{p_{1}}+\frac{1}{p_{2}}\right)  -1}\left \Vert f\right \Vert
_{L_{p}\left(  B(x_{0},t)\right)  }dt.
\end{align*}
Thus, putting estimates $F$ and $G$ together, we get the desired conclusion%
\begin{align*}
\left \Vert \lbrack \left(  b_{1},b_{2}\right)  ,T_{\Omega,\alpha}]f\right \Vert
_{L_{q_{1}}\left(  B(x_{0},r)\right)  }  & \lesssim \Vert b_{1}\Vert
_{LC_{p_{1},\lambda_{1}}^{\left \{  x_{0}\right \}  }}\Vert b_{2}\Vert
_{LC_{p_{2},\lambda_{2}}^{\left \{  x_{0}\right \}  }}r^{\frac{n}{q_{1}}}\\
& \times \int \limits_{2r}^{\infty}\left(  1+\ln \frac{t}{r}\right)
^{2}t^{-\frac{n}{q_{1}}+n\left(  \lambda_{1}+\lambda_{2}\right)  +n\left(
\frac{1}{p_{1}}+\frac{1}{p_{2}}\right)  -1}\left \Vert f\right \Vert
_{L_{p}\left(  B(x_{0},t)\right)  }dt.
\end{align*}
For the case of $q_{1}<s$, we can also use the same method, so we omit the
details. This completes the proof of Lemma \ref{lemma2}.
\end{proof}

At last, throughout the paper we use the letter $C$ for a positive constant,
independent of appropriate parameters and not necessarily the same at each
occurrence. By $A\lesssim B$ we mean that $A\leq CB$ with some positive
constant $C$ independent of appropriate quantities. If $A\lesssim B$ and
$B\lesssim A$, we write $A\approx B$ and say that $A$ and $B$ are equivalent.

\section{Main Results}

Now we are ready to give the following main results with their proofs , respectively.

\begin{theorem}
\label{teo4}Suppose that $x_{0}\in{\mathbb{R}^{n}}$, $\Omega \in L_{\infty
}({\mathbb{R}^{n}})\times L_{s}(S^{n-1})$, $s>1$, is homogeneous of degree
zero with respect to the second variable $y$ on ${\mathbb{R}^{n}}$. Let
$T_{\Omega,\alpha}$ be a linear operator satisfying condition (\ref{e1}). Let
also $0<\alpha<n$ and $1<q,q_{1},p_{i},p<\frac{n}{\alpha}$ with $\frac{1}%
{q}=\sum \limits_{i=1}^{m}\frac{1}{p_{i}}+\frac{1}{p}$, $\frac{1}{q_{1}}%
=\frac{1}{q}-\frac{\alpha}{n}$ and $\overrightarrow{b}\in LC_{p_{i}%
,\lambda_{i}}^{\left \{  x_{0}\right \}  }({\mathbb{R}^{n}})$ for $0\leq
\lambda_{i}<\frac{1}{n}$, $i=1,\ldots,m$.

Let also, for $s^{\prime}\leq q$ the pair $(\varphi_{1},\varphi_{2})$
satisfies the condition%
\begin{equation}%
%TCIMACRO{\dint \limits_{r}^{\infty}}%
%BeginExpansion
{\displaystyle \int \limits_{r}^{\infty}}
%EndExpansion
\left(  1+\ln \frac{t}{r}\right)  ^{m}\frac{\operatorname*{essinf}%
\limits_{t<\tau<\infty}\varphi_{1}(x_{0},\tau)\tau^{\frac{n}{p}}}{t^{n\left(
\frac{1}{q_{1}}-\left(
%TCIMACRO{\dsum \limits_{i=1}^{m}}%
%BeginExpansion
{\displaystyle \sum \limits_{i=1}^{m}}
%EndExpansion
\lambda_{i}+%
%TCIMACRO{\dsum \limits_{i=1}^{m}}%
%BeginExpansion
{\displaystyle \sum \limits_{i=1}^{m}}
%EndExpansion
\frac{1}{p_{i}}\right)  \right)  +1}}\leq C\, \varphi_{2}(x_{0},r),\label{47}%
\end{equation}
and for $q_{1}<s$ the pair $(\varphi_{1},\varphi_{2})$ satisfies the condition%
\begin{equation}
\int \limits_{r}^{\infty}\left(  1+\ln \frac{t}{r}\right)  ^{m}\frac
{\operatorname*{essinf}\limits_{t<\tau<\infty}\varphi_{1}(x_{0},\tau
)\tau^{\frac{n}{p}}}{t^{n\left(  \frac{1}{q_{1}}-\left(  \frac{1}{s}+%
%TCIMACRO{\dsum \limits_{i=1}^{m}}%
%BeginExpansion
{\displaystyle \sum \limits_{i=1}^{m}}
%EndExpansion
\lambda_{i}+%
%TCIMACRO{\dsum \limits_{i=1}^{m}}%
%BeginExpansion
{\displaystyle \sum \limits_{i=1}^{m}}
%EndExpansion
\frac{1}{p_{i}}\right)  \right)  +1}}dt\leq C\, \varphi_{2}(x_{0}%
,r)r^{\frac{n}{s}},\label{48}%
\end{equation}
where $C$ does not depend on $r$.

Then, the operator $[\overrightarrow{b},T_{\Omega,\alpha}]$ is bounded from
$LM_{p,\varphi_{1}}^{\{x_{0}\}}$ to $LM_{q_{1},\varphi_{2}}^{\{x_{0}\}}$.
Moreover,
\[
\left \Vert \lbrack \overrightarrow{b},T_{\Omega,\alpha}]\right \Vert
_{LM_{q_{1},\varphi_{2}}^{\{x_{0}\}}}\lesssim%
%TCIMACRO{\dprod \limits_{i=1}^{m}}%
%BeginExpansion
{\displaystyle \prod \limits_{i=1}^{m}}
%EndExpansion
\Vert \overrightarrow{b}\Vert_{LC_{p_{i},\lambda_{i}}^{\left \{  x_{0}\right \}
}}\left \Vert f\right \Vert _{LM_{p,\varphi_{1}}^{\{x_{0}\}}}.
\]

\end{theorem}

\begin{proof}
Since $f\in LM_{p_{1},\varphi_{1}}^{\{x_{0}\}}$, by (\ref{5}) and it is also
non-decreasing, with respect to $t$, of the norm $\left \Vert f\right \Vert
_{L_{p_{1}}\left(  B\left(  x_{0},t\right)  \right)  }$, we get%
\begin{align}
& \frac{\left \Vert f\right \Vert _{L_{p}\left(  B\left(  x_{0},t\right)
\right)  }}{\operatorname*{essinf}\limits_{0<t<\tau<\infty}\varphi_{1}%
(x_{0},\tau)\tau^{\frac{n}{p}}}\leq \operatorname*{esssup}\limits_{0<t<\tau
<\infty}\frac{\left \Vert f\right \Vert _{L_{p}\left(  B\left(  x_{0},t\right)
\right)  }}{\varphi_{1}(x_{0},\tau)\tau^{\frac{n}{p}}}\nonumber \\
& \leq \operatorname*{esssup}\limits_{0<\tau<\infty}\frac{\left \Vert
f\right \Vert _{L_{p}\left(  B\left(  x_{0},\tau \right)  \right)  }}%
{\varphi_{1}(x_{0},\tau)\tau^{\frac{n}{p}}}\leq \left \Vert f\right \Vert
_{LM_{p,\varphi_{1}}^{\{x_{0}\}}}.\label{9}%
\end{align}
For $s^{\prime}\leq q<\infty$, since $(\varphi_{1},\varphi_{2})$ satisfies
(\ref{47}) and by (\ref{9}), we have%
\begin{align}
&
%TCIMACRO{\dint \limits_{r}^{\infty}}%
%BeginExpansion
{\displaystyle \int \limits_{r}^{\infty}}
%EndExpansion
\left(  1+\ln \frac{t}{r}\right)  ^{m}t^{n\left(  -\frac{1}{q_{1}}+\left(
%TCIMACRO{\dsum \limits_{i=1}^{m}}%
%BeginExpansion
{\displaystyle \sum \limits_{i=1}^{m}}
%EndExpansion
\lambda_{i}+%
%TCIMACRO{\dsum \limits_{i=1}^{m}}%
%BeginExpansion
{\displaystyle \sum \limits_{i=1}^{m}}
%EndExpansion
\frac{1}{p_{i}}\right)  \right)  -1}\Vert f\Vert_{L_{p}(B(x_{0},t))}%
dt\nonumber \\
& \leq \int \limits_{r}^{\infty}\left(  1+\ln \frac{t}{r}\right)  ^{m}%
\frac{\left \Vert f\right \Vert _{L_{p}\left(  B\left(  x_{0},t\right)  \right)
}}{\operatorname*{essinf}\limits_{t<\tau<\infty}\varphi_{1}(x_{0},\tau
)\tau^{\frac{n}{p}}}\frac{\operatorname*{essinf}\limits_{t<\tau<\infty}%
\varphi_{1}(x_{0},\tau)\tau^{\frac{n}{p}}}{t^{n\left(  \frac{1}{q_{1}}-\left(
%
%TCIMACRO{\dsum \limits_{i=1}^{m}}%
%BeginExpansion
{\displaystyle \sum \limits_{i=1}^{m}}
%EndExpansion
\lambda_{i}+%
%TCIMACRO{\dsum \limits_{i=1}^{m}}%
%BeginExpansion
{\displaystyle \sum \limits_{i=1}^{m}}
%EndExpansion
\frac{1}{p_{i}}\right)  \right)  +1}}dt\nonumber \\
& \leq C\left \Vert f\right \Vert _{LM_{p,\varphi_{1}}^{\{x_{0}\}}}%
\int \limits_{r}^{\infty}\left(  1+\ln \frac{t}{r}\right)  ^{m}\frac
{\operatorname*{essinf}\limits_{t<\tau<\infty}\varphi_{1}(x_{0},\tau
)\tau^{\frac{n}{p}}}{t^{n\left(  \frac{1}{q_{1}}-\left(
%TCIMACRO{\dsum \limits_{i=1}^{m}}%
%BeginExpansion
{\displaystyle \sum \limits_{i=1}^{m}}
%EndExpansion
\lambda_{i}+%
%TCIMACRO{\dsum \limits_{i=1}^{m}}%
%BeginExpansion
{\displaystyle \sum \limits_{i=1}^{m}}
%EndExpansion
\frac{1}{p_{i}}\right)  \right)  +1}}dt\nonumber \\
& \leq C\left \Vert f\right \Vert _{LM_{p,\varphi_{1}}^{\{x_{0}\}}}\varphi
_{2}(x_{0},r).\label{200*}%
\end{align}
Then by (\ref{200}) and (\ref{200*}), we get%
\begin{align*}
\left \Vert \lbrack \overrightarrow{b},T_{\Omega,\alpha}]f\right \Vert
_{LM_{q_{1},\varphi_{2}}^{\{x_{0}\}}}  & =\sup_{r>0}\varphi_{2}\left(
x_{0},r\right)  ^{-1}|B(x_{0},r)|^{-\frac{1}{q_{1}}}\left \Vert
[\overrightarrow{b},T_{\Omega,\alpha}]f\right \Vert _{L_{q_{1}}\left(  B\left(
x_{0},r\right)  \right)  }\\
& \leq C%
%TCIMACRO{\dprod \limits_{i=1}^{m}}%
%BeginExpansion
{\displaystyle \prod \limits_{i=1}^{m}}
%EndExpansion
\Vert \overrightarrow{b}\Vert_{LC_{p_{i},\lambda_{i}}^{\left \{  x_{0}\right \}
}}\sup_{r>0}\varphi_{2}\left(  x_{0},r\right)  ^{-1}\\
& \times%
%TCIMACRO{\dint \limits_{r}^{\infty}}%
%BeginExpansion
{\displaystyle \int \limits_{r}^{\infty}}
%EndExpansion
\left(  1+\ln \frac{t}{r}\right)  ^{m}t^{n\left(  -\frac{1}{q_{1}}+\left(
%TCIMACRO{\dsum \limits_{i=1}^{m}}%
%BeginExpansion
{\displaystyle \sum \limits_{i=1}^{m}}
%EndExpansion
\lambda_{i}+%
%TCIMACRO{\dsum \limits_{i=1}^{m}}%
%BeginExpansion
{\displaystyle \sum \limits_{i=1}^{m}}
%EndExpansion
\frac{1}{p_{i}}\right)  \right)  -1}\Vert f\Vert_{L_{p}(B(x_{0},t))}dt\\
& \leq C%
%TCIMACRO{\dprod \limits_{i=1}^{m}}%
%BeginExpansion
{\displaystyle \prod \limits_{i=1}^{m}}
%EndExpansion
\Vert \overrightarrow{b}\Vert_{LC_{p_{i},\lambda_{i}}^{\left \{  x_{0}\right \}
}}\left \Vert f\right \Vert _{LM_{p,\varphi_{1}}^{\{x_{0}\}}}.
\end{align*}
For the case of $q_{1}<s$, we can also use the same method, so we omit the
details. Thus, we finish the proof of Theorem \ref{teo4}.
\end{proof}

\begin{corollary}
Suppose that $x_{0}\in{\mathbb{R}^{n}}$, $\Omega \in L_{\infty}({\mathbb{R}%
^{n}})\times L_{s}(S^{n-1})$, $s>1$, is homogeneous of degree zero with
respect to the second variable $y$ on ${\mathbb{R}^{n}}$. Let $0<\alpha<n$ and
$1<q,q_{1},p_{i},p<\frac{n}{\alpha}$ with $\frac{1}{q}=\sum \limits_{i=1}%
^{m}\frac{1}{p_{i}}+\frac{1}{p}$, $\frac{1}{q_{1}}=\frac{1}{q}-\frac{\alpha
}{n}$ and $\overrightarrow{b}\in LC_{p_{i},\lambda_{i}}^{\left \{
x_{0}\right \}  }({\mathbb{R}^{n}})$ for $0\leq \lambda_{i}<\frac{1}{n}$,
$i=1,\ldots,m$. Let also, for $s^{\prime}\leq q$ the pair $(\varphi
_{1},\varphi_{2})$ satisfy condition (\ref{47}) and for $q_{1}<s$ the pair
$(\varphi_{1},\varphi_{2})$ satisfy condition (\ref{48}). Then, the operators
$M_{\Omega,\overrightarrow{b},\alpha}$ and $[\overrightarrow{b},\overline
{T}_{\Omega,\alpha}]$ are bounded from $LM_{p,\varphi_{1}}^{\{x_{0}\}}$ to
$LM_{q_{1},\varphi_{2}}^{\{x_{0}\}}$.
\end{corollary}

\begin{corollary}
\cite{Gurbuz, Gurbuz3} Suppose that $x_{0}\in{\mathbb{R}^{n}}$, $\Omega \in
L_{\infty}({\mathbb{R}^{n}})\times L_{s}(S^{n-1})$, $s>1$, is homogeneous of
degree zero with respect to the second variable $y$ on ${\mathbb{R}^{n}}$. Let
$T_{\Omega,\alpha}$ be a linear operator satisfying condition (\ref{e1}) and
bounded from $L_{p}({\mathbb{R}^{n}})$ to $L_{q}({\mathbb{R}^{n}})$. Let
$0<\alpha<n$, $1<p<\frac{n}{\alpha}$, $b\in LC_{p_{2},\lambda}^{\left \{
x_{0}\right \}  }\left(
%TCIMACRO{\U{211d} }%
%BeginExpansion
\mathbb{R}
%EndExpansion
^{n}\right)  $, $0\leq \lambda<\frac{1}{n}$, $\frac{1}{p}=\frac{1}{p_{1}}%
+\frac{1}{p_{2}}$, $\frac{1}{q}=\frac{1}{p}-\frac{\alpha}{n}$, $\frac{1}%
{q_{1}}=\frac{1}{p_{1}}-\frac{\alpha}{n}$.

Let also, for $s^{\prime}\leq p$ the pair $(\varphi_{1},\varphi_{2})$ satisfy
the condition%
\[
\int \limits_{r}^{\infty}\left(  1+\ln \frac{t}{r}\right)  \frac
{\operatorname*{essinf}\limits_{t<\tau<\infty}\varphi_{1}(x_{0},\tau
)\tau^{\frac{n}{p_{1}}}}{t^{\frac{n}{q_{1}}+1-n\lambda}}dt\leq C\, \varphi
_{2}(x_{0},r),
\]
and for $q_{1}<s$ the pair $(\varphi_{1},\varphi_{2})$ satisfy the condition%
\[
\int \limits_{r}^{\infty}\left(  1+\ln \frac{t}{r}\right)  \frac
{\operatorname*{essinf}\limits_{t<\tau<\infty}\varphi_{1}(x_{0},\tau
)\tau^{\frac{n}{p_{1}}}}{t^{\frac{n}{q_{1}}-\frac{n}{s}+1-n\lambda}}dt\leq C\,
\varphi_{2}(x_{0},r)r^{\frac{n}{s}},
\]
where $C$ does not depend on $r$.

Then, the operator $[b,T_{\Omega,\alpha}]$ is bounded from $LM_{p_{1}%
,\varphi_{1}}^{\{x_{0}\}}$ to $LM_{q,\varphi_{2}}^{\{x_{0}\}}$. Moreover,
\[
\left \Vert \lbrack b,T_{\Omega,\alpha}]\right \Vert _{LM_{q,\varphi_{2}%
}^{\{x_{0}\}}}\lesssim \left \Vert b\right \Vert _{LC_{p_{2},\lambda}^{\left \{
x_{0}\right \}  }}\left \Vert f\right \Vert _{LM_{p_{1},\varphi_{1}}^{\{x_{0}\}}%
}.
\]

\end{corollary}

\begin{theorem}
\label{teo5}Suppose that $x_{0}\in{\mathbb{R}^{n}}$, $\Omega \in L_{\infty
}({\mathbb{R}^{n}})\times L_{s}(S^{n-1})$, $s>1$, is homogeneous of degree
zero with respect to the second variable $y$ on ${\mathbb{R}^{n}}$. Let
$T_{\Omega,\alpha}$ be a linear operator satisfying condition (\ref{e1}). Let
also $0<\alpha<n$ and $1<q,q_{1},p_{i},p<\frac{n}{\alpha}$ with $\frac{1}%
{q}=\sum \limits_{i=1}^{m}\frac{1}{p_{i}}+\frac{1}{p}$, $\frac{1}{q_{1}}%
=\frac{1}{q}-\frac{\alpha}{n}$ and $\overrightarrow{b}\in LC_{p_{i}%
,\lambda_{i}}^{\left \{  x_{0}\right \}  }({\mathbb{R}^{n}})$ for $0\leq
\lambda_{i}<\frac{1}{n}$, $i=1,\ldots,m$. Let for $s^{\prime}\leq q$ the pair
$(\varphi_{1},\varphi_{2})$ satisfies conditions (\ref{2})-(\ref{3}) and%
\begin{equation}%
%TCIMACRO{\dint \limits_{r}^{\infty}}%
%BeginExpansion
{\displaystyle \int \limits_{r}^{\infty}}
%EndExpansion
\left(  1+\ln \frac{t}{r}\right)  ^{m}\varphi_{1}(x_{0},t)\frac{t^{\frac{n}{p}%
}}{t^{n\left(  \frac{1}{q_{1}}-\left(
%TCIMACRO{\dsum \limits_{i=1}^{m}}%
%BeginExpansion
{\displaystyle \sum \limits_{i=1}^{m}}
%EndExpansion
\lambda_{i}+%
%TCIMACRO{\dsum \limits_{i=1}^{m}}%
%BeginExpansion
{\displaystyle \sum \limits_{i=1}^{m}}
%EndExpansion
\frac{1}{p_{i}}\right)  \right)  +1}}\leq C_{0}\, \varphi_{2}(x_{0}%
,r),\label{10*}%
\end{equation}
where $C_{0}$ does not depend on $r>0$,%
\begin{equation}
\lim_{r\rightarrow0}\frac{\ln \frac{1}{r}}{\varphi_{2}(x_{0},r)}=0\label{11*}%
\end{equation}
and%
\begin{equation}
c_{\delta}:=%
%TCIMACRO{\dint \limits_{\delta}^{\infty}}%
%BeginExpansion
{\displaystyle \int \limits_{\delta}^{\infty}}
%EndExpansion
\left(  1+\ln \left \vert t\right \vert \right)  ^{m}\varphi_{1}\left(
x_{0},t\right)  \frac{t^{\frac{n}{p}}}{t^{n\left(  \frac{1}{q_{1}}-\left(
%TCIMACRO{\dsum \limits_{i=1}^{m}}%
%BeginExpansion
{\displaystyle \sum \limits_{i=1}^{m}}
%EndExpansion
\lambda_{i}+%
%TCIMACRO{\dsum \limits_{i=1}^{m}}%
%BeginExpansion
{\displaystyle \sum \limits_{i=1}^{m}}
%EndExpansion
\frac{1}{p_{i}}\right)  \right)  +1}}dt<\infty \label{12*}%
\end{equation}
for every $\delta>0$, and for $q_{1}<s$ the pair $(\varphi_{1},\varphi_{2})$
satisfies conditions (\ref{2})-(\ref{3}) and also%
\begin{equation}%
%TCIMACRO{\dint \limits_{r}^{\infty}}%
%BeginExpansion
{\displaystyle \int \limits_{r}^{\infty}}
%EndExpansion
\left(  1+\ln \frac{t}{r}\right)  ^{m}\varphi_{1}(x_{0},t)\frac{t^{\frac{n}{p}%
}}{t^{n\left(  \frac{1}{q_{1}}-\left(  \frac{1}{s}+%
%TCIMACRO{\dsum \limits_{i=1}^{m}}%
%BeginExpansion
{\displaystyle \sum \limits_{i=1}^{m}}
%EndExpansion
\lambda_{i}+%
%TCIMACRO{\dsum \limits_{i=1}^{m}}%
%BeginExpansion
{\displaystyle \sum \limits_{i=1}^{m}}
%EndExpansion
\frac{1}{p_{i}}\right)  \right)  +1}}\leq C\, \varphi_{2}(x_{0},r)r^{\frac
{n}{s}},\label{13*}%
\end{equation}
where $C_{0}$ does not depend on $r>0$,%
\[
\lim_{r\rightarrow0}\frac{\ln \frac{1}{r}}{\varphi_{2}(x_{0},r)}=0
\]
and%
\begin{equation}
c_{\delta^{\prime}}:=%
%TCIMACRO{\dint \limits_{\delta^{\prime}}^{\infty}}%
%BeginExpansion
{\displaystyle \int \limits_{\delta^{\prime}}^{\infty}}
%EndExpansion
\left(  1+\ln \left \vert t\right \vert \right)  ^{m}\varphi_{1}\left(
x_{0},t\right)  \frac{t^{\frac{n}{p}}}{t^{n\left(  \frac{1}{q_{1}}-\left(
\frac{1}{s}+%
%TCIMACRO{\dsum \limits_{i=1}^{m}}%
%BeginExpansion
{\displaystyle \sum \limits_{i=1}^{m}}
%EndExpansion
\lambda_{i}+%
%TCIMACRO{\dsum \limits_{i=1}^{m}}%
%BeginExpansion
{\displaystyle \sum \limits_{i=1}^{m}}
%EndExpansion
\frac{1}{p_{i}}\right)  \right)  +1}}dt<\infty \label{14*}%
\end{equation}
for every $\delta^{\prime}>0$.

Then the operator $[\overrightarrow{b},T_{\Omega,\alpha}]$ is bounded from
$VLM_{p,\varphi_{1}}^{\left \{  x_{0}\right \}  }$ to$VLM_{q_{1},\varphi_{2}%
}^{\left \{  x_{0}\right \}  }$. Moreover,%
\begin{equation}
\left \Vert \lbrack \overrightarrow{b},T_{\Omega,\alpha}]f\right \Vert
_{VLM_{q_{1},\varphi_{2}}^{\left \{  x_{0}\right \}  }}\lesssim%
%TCIMACRO{\dprod \limits_{i=1}^{m}}%
%BeginExpansion
{\displaystyle \prod \limits_{i=1}^{m}}
%EndExpansion
\Vert \overrightarrow{b}\Vert_{LC_{p_{i},\lambda_{i}}^{\left \{  x_{0}\right \}
}}\left \Vert f\right \Vert _{VLM_{p,\varphi_{1}}^{\left \{  x_{0}\right \}  }%
}.\label{15*}%
\end{equation}

\end{theorem}

\begin{proof}
Since the inequality (\ref{15*}) holds by Theorem \ref{teo4}, we only have to
prove the implication%
\begin{equation}
\lim \limits_{r\rightarrow0}\frac{r^{-\frac{n}{p}}\Vert f\Vert_{L_{p}%
(B(x_{0},r))}}{\varphi_{1}(x_{0},r)}=0\text{ implies }\lim
\limits_{r\rightarrow0}\frac{r^{-\frac{n}{q_{1}}}\left \Vert [\overrightarrow
{b},T_{\Omega,\alpha}]f\right \Vert _{L_{q_{1}}\left(  B\left(  x_{0},r\right)
\right)  }}{\varphi_{2}(x_{0},r)}=0.\label{16*}%
\end{equation}

To show that%
\[
\frac{r^{-\frac{n}{q_{1}}}\left \Vert [\overrightarrow{b},T_{\Omega,\alpha
}]f\right \Vert _{L_{q_{1}}\left(  B\left(  x_{0},r\right)  \right)  }}%
{\varphi_{2}(x_{0},r)}<\epsilon \text{ for small }r,
\]
we use the estimate (\ref{200}):%
\[
\frac{r^{-\frac{n}{q_{1}}}\left \Vert [\overrightarrow{b},T_{\Omega,\alpha
}]f\right \Vert _{L_{q_{1}}\left(  B\left(  x_{0},r\right)  \right)  }}%
{\varphi_{2}(x_{0},r)}\lesssim \frac{%
%TCIMACRO{\dprod \limits_{i=1}^{m}}%
%BeginExpansion
{\displaystyle \prod \limits_{i=1}^{m}}
%EndExpansion
\Vert \overrightarrow{b}\Vert_{LC_{p_{i},\lambda_{i}}^{\left \{  x_{0}\right \}
}}}{\varphi_{2}(x_{0},r)}\,
%TCIMACRO{\dint \limits_{r}^{\infty}}%
%BeginExpansion
{\displaystyle \int \limits_{r}^{\infty}}
%EndExpansion
\left(  1+\ln \frac{t}{r}\right)  ^{m}t^{n\left(  -\frac{1}{q_{1}}+\left(
%TCIMACRO{\dsum \limits_{i=1}^{m}}%
%BeginExpansion
{\displaystyle \sum \limits_{i=1}^{m}}
%EndExpansion
\lambda_{i}+%
%TCIMACRO{\dsum \limits_{i=1}^{m}}%
%BeginExpansion
{\displaystyle \sum \limits_{i=1}^{m}}
%EndExpansion
\frac{1}{p_{i}}\right)  \right)  -1}\Vert f\Vert_{L_{p}(B(x_{0},t))}dt.
\]

We take $r<\delta_{0}$, where $\delta_{0}$ is small enough and split the integration:%

\begin{equation}
\frac{r^{-\frac{n}{q_{1}}}\left \Vert [\overrightarrow{b},T_{\Omega,\alpha
}]f\right \Vert _{L_{q_{1}}\left(  B\left(  x_{0},r\right)  \right)  }}%
{\varphi_{2}(x_{0},r)}\leq C\left[  I_{\delta_{0}}\left(  x_{0},r\right)
+J_{\delta_{0}}\left(  x_{0},r\right)  \right]  ,\label{17*}%
\end{equation}
where $\delta_{0}>0$ (we may take $\delta_{0}<1$), and
\[
I_{\delta_{0}}\left(  x_{0},r\right)  :=\frac{1}{\varphi_{2}(x_{0},r)}%
%TCIMACRO{\dint \limits_{r}^{\delta_{0}}}%
%BeginExpansion
{\displaystyle \int \limits_{r}^{\delta_{0}}}
%EndExpansion
\left(  1+\ln \frac{t}{r}\right)  ^{m}t^{n\left(  -\frac{1}{q_{1}}+\left(
%TCIMACRO{\dsum \limits_{i=1}^{m}}%
%BeginExpansion
{\displaystyle \sum \limits_{i=1}^{m}}
%EndExpansion
\lambda_{i}+%
%TCIMACRO{\dsum \limits_{i=1}^{m}}%
%BeginExpansion
{\displaystyle \sum \limits_{i=1}^{m}}
%EndExpansion
\frac{1}{p_{i}}\right)  \right)  -1}\Vert f\Vert_{L_{p}(B(x_{0},t))}dt,
\]
and%
\[
J_{\delta_{0}}\left(  x_{0},r\right)  :=\frac{1}{\varphi_{2}(x_{0},r)}%
%TCIMACRO{\dint \limits_{\delta_{0}}^{\infty}}%
%BeginExpansion
{\displaystyle \int \limits_{\delta_{0}}^{\infty}}
%EndExpansion
\left(  1+\ln \frac{t}{r}\right)  ^{m}t^{n\left(  -\frac{1}{q_{1}}+\left(
%TCIMACRO{\dsum \limits_{i=1}^{m}}%
%BeginExpansion
{\displaystyle \sum \limits_{i=1}^{m}}
%EndExpansion
\lambda_{i}+%
%TCIMACRO{\dsum \limits_{i=1}^{m}}%
%BeginExpansion
{\displaystyle \sum \limits_{i=1}^{m}}
%EndExpansion
\frac{1}{p_{i}}\right)  \right)  -1}\Vert f\Vert_{L_{p}(B(x_{0},t))}dt
\]
and $r<\delta_{0}$. Now we can choose any fixed $\delta_{0}>0$ such that%
\[
\frac{t^{-\frac{n}{p}}\left \Vert f\right \Vert _{L_{p}\left(  B\left(
x_{0},t\right)  \right)  }}{\varphi_{1}(x_{0},t)}<\frac{\epsilon}{2CC_{0}%
},\qquad t\leq \delta_{0},
\]
where $C$ and $C_{0}$ are constants from (\ref{10*}) and (\ref{17*}). This
allows to estimate the first term uniformly in $r\in \left(  0,\delta
_{0}\right)  $:%
\[
CI_{\delta_{0}}\left(  x_{0},r\right)  <\frac{\epsilon}{2},\qquad
0<r<\delta_{0}.
\]
For the second term, writing $1+\ln \frac{t}{r}\leq1+\left \vert \ln
t\right \vert +\ln \frac{1}{r}$, we obtain%
\[
J_{\delta_{0}}\left(  x_{0},r\right)  \leq \frac{c_{\delta_{0}}+\widetilde
{c_{\delta_{0}}}\ln \frac{1}{r}}{\varphi_{2}(x_{0},r)}\left \Vert f\right \Vert
_{LM_{p,\varphi_{1}}^{\left \{  x_{0}\right \}  }},
\]
where $c_{\delta_{0}}$ is the constant from (\ref{12*}) with $\delta
=\delta_{0}$ and $\widetilde{c_{\delta_{0}}}$ is a similar constant with
omitted logarithmic factor in the integrand. Then, by (\ref{11*}) we can
choose small enough $r$ such that%
\[
J_{\delta_{0}}\left(  x_{0},r\right)  <\frac{\epsilon}{2},
\]
which completes the proof of (\ref{16*}).

For the case of $q_{1}<s$, we can also use the same method, to obtain the
desired result. Therefore, the proof of Theorem \ref{teo5} is completed.
\end{proof}

\begin{remark}
Conditions (\ref{12*}) and (\ref{14*}) are not needed in the case when
$\varphi(x_{0},r)$ does not depend on $x_{0}$, since (\ref{12*}) follows from
(\ref{10*}) and similarly, (\ref{14*}) follows from (\ref{13*}) in this case.
\end{remark}

\begin{corollary}
Suppose that $x_{0}\in{\mathbb{R}^{n}}$, $\Omega \in L_{\infty}({\mathbb{R}%
^{n}})\times L_{s}(S^{n-1})$, $s>1$, is homogeneous of degree zero with
respect to the second variable $y$ on ${\mathbb{R}^{n}}$. Let $0<\alpha<n$ and
$1<q,q_{1},p_{i},p<\frac{n}{\alpha}$ with $\frac{1}{q}=\sum \limits_{i=1}%
^{m}\frac{1}{p_{i}}+\frac{1}{p}$, $\frac{1}{q_{1}}=\frac{1}{q}-\frac{\alpha
}{n}$ and $\overrightarrow{b}\in LC_{p_{i},\lambda_{i}}^{\left \{
x_{0}\right \}  }({\mathbb{R}^{n}})$ for $0\leq \lambda_{i}<\frac{1}{n}$,
$i=1,\ldots,m$. Let also, for $s^{\prime}\leq q$ the pair $\left(  \varphi
_{1},\varphi_{2}\right)  $ satisfies conditions (\ref{2})-(\ref{3}%
)-(\ref{11*}) and (\ref{10*})-(\ref{12*}) and for $q_{1}<s$ the pair $\left(
\varphi_{1},\varphi_{2}\right)  $ satisfies conditions (\ref{2})-(\ref{3}%
)-(\ref{11*}) and (\ref{13*})-(\ref{14*}). Then, the operators $M_{\Omega
,\overrightarrow{b},\alpha}$ and $[\overrightarrow{b},\overline{T}%
_{\Omega,\alpha}]$ are bounded from $VLM_{p,\varphi_{1}}^{\left \{
x_{0}\right \}  }$ to $VLM_{q_{1},\varphi_{2}}^{\left \{  x_{0}\right \}  }$.
\end{corollary}

\begin{corollary}
Suppose that $x_{0}\in{\mathbb{R}^{n}}$, $\Omega \in L_{\infty}({\mathbb{R}%
^{n}})\times L_{s}(S^{n-1})$, $s>1$, is homogeneous of degree zero with
respect to the second variable $y$ on ${\mathbb{R}^{n}}$. Let $0<\alpha<n$,
$1<p<\frac{n}{\alpha}$, $b\in LC_{p_{2},\lambda}^{\left \{  x_{0}\right \}
}\left(
%TCIMACRO{\U{211d} }%
%BeginExpansion
\mathbb{R}
%EndExpansion
^{n}\right)  $, $0\leq \lambda<\frac{1}{n}$, $\frac{1}{p}=\frac{1}{p_{1}}%
+\frac{1}{p_{2}}$, $\frac{1}{q}=\frac{1}{p}-\frac{\alpha}{n}$, $\frac{1}%
{q_{1}}=\frac{1}{p_{1}}-\frac{\alpha}{n}$, and $T_{\Omega,\alpha}$ is a linear
operator satisfying condition (\ref{e1}) and bounded from $L_{p}%
({\mathbb{R}^{n}})$ to $L_{q}({\mathbb{R}^{n}})$. Let also, for $s^{\prime
}\leq p$ the pair $(\varphi_{1},\varphi_{2})$ satisfies conditions
(\ref{2})-(\ref{3}) and%
\[
\int \limits_{r}^{\infty}\left(  1+\ln \frac{t}{r}\right)  \varphi_{1}\left(
x_{0},t\right)  \frac{t^{\frac{n}{p_{1}}}}{t^{\frac{n}{q_{1}}+1-n\lambda}%
}dt\leq C_{0}\varphi_{2}\left(  x_{0},r\right)  ,
\]
where $C_{0}$ does not depend on $r>0$,%
\[
\lim_{r\rightarrow0}\frac{\ln \frac{1}{r}}{\varphi_{2}(x_{0},r)}=0
\]
and
\[
c_{\delta}:=%
%TCIMACRO{\dint \limits_{\delta}^{\infty}}%
%BeginExpansion
{\displaystyle \int \limits_{\delta}^{\infty}}
%EndExpansion
\left(  1+\ln \left \vert t\right \vert \right)  \varphi_{1}\left(
x_{0},t\right)  \frac{t^{\frac{n}{p_{1}}}}{t^{\frac{n}{q_{1}}+1-n\lambda}%
}dt<\infty
\]
for every $\delta>0$, and for $q_{1}<s$ the pair $(\varphi_{1},\varphi_{2})$
satisfies conditions (\ref{2})-(\ref{3}) and also%
\[
\int \limits_{r}^{\infty}\left(  1+\ln \frac{t}{r}\right)  \varphi_{1}\left(
x_{0},t\right)  \frac{t^{\frac{n}{p_{1}}}}{t^{\frac{n}{q_{1}}-\frac{n}%
{s}+1-n\lambda}}dt\leq C_{0}\varphi_{2}(x_{0},r)r^{\frac{n}{s}},
\]
where $C_{0}$ does not depend on $r>0$,%
\[
\lim_{r\rightarrow0}\frac{\ln \frac{1}{r}}{\varphi_{2}(x_{0},r)}=0
\]
and%
\[
c_{\delta^{\prime}}:=%
%TCIMACRO{\dint \limits_{\delta^{\prime}}^{\infty}}%
%BeginExpansion
{\displaystyle \int \limits_{\delta^{\prime}}^{\infty}}
%EndExpansion
\left(  1+\ln \left \vert t\right \vert \right)  \varphi_{1}\left(
x_{0},t\right)  \frac{t^{\frac{n}{p_{1}}}}{t^{\frac{n}{q_{1}}-\frac{n}%
{s}+1-n\lambda}}dt<\infty
\]
for every $\delta^{\prime}>0$.

Then the operator $[b,T_{\Omega,\alpha}]$ is bounded from $VLM_{p_{1}%
,\varphi_{1}}^{\left \{  x_{0}\right \}  }$ to $VLM_{q,\varphi_{2}}^{\left \{
x_{0}\right \}  }$. Furthermore, we have%
\[
\left \Vert \lbrack b,T_{\Omega,\alpha}]f\right \Vert _{VLM_{q,\varphi_{2}%
}^{\left \{  x_{0}\right \}  }}\lesssim \left \Vert b\right \Vert _{LC_{p_{2}%
,\lambda}^{\left \{  x_{0}\right \}  }}\left \Vert f\right \Vert _{VLM_{p_{1}%
,\varphi_{1}}^{\left \{  x_{0}\right \}  }}.
\]

\end{corollary}


\begin{thebibliography}{99}                                                                                               %
\bibitem {Aos1}A. Abdalmonem, O. Abdalrhman, S. Tao, The boundedness of
fractional integral with variable kernel on variable exponent Herz-Morrey
spaces, J. Appl. Math. Phys., 4 (2016), 787-795.

\bibitem {Aos2}A. Abdalmonem, O. Abdalrhman, S. Tao, Boundedness of fractional
integral with variable kernel and their commutators on variable exponent Herz
spaces, Appl. Math. 7 (2016), 1165-1182.

\bibitem {BGGS}A.S. Balakishiyev, V.S. Guliyev, F. Gurbuz and A. Serbetci,
Sublinear operators with rough kernel generated by Calderon-Zygmund operators
and their commutators on generalized local Morrey spaces, J. Inequal. Appl.
2015, 2015:61. doi:10.1186/s13660-015-0582-y.

\bibitem {Caf}L. Caffarelli, Elliptic second order equations, Rend. Semin.
Math. Fis. Milano, 58 (1990), 253-284.

\bibitem {Cao-Chen}X.N. Cao, D.X. Chen, The boundedness of Toeplitz-type
operators on vanishing Morrey spaces. Anal. Theory Appl. 27 (2011), 309-319.

\bibitem {Chanillo}S. Chanillo, A note on commutators,\ Indiana Univ. Math.
J., 31 (1), (1982 ), 7-16.

\bibitem {Chen}Y. Chen, Y. Ding, R. Li, The boundedness for commutator of
fractional integral operator with rough variable kernel, Potential Anal., 38
(2013), 119-142.

\bibitem {ChFraL1}F. Chiarenza, M. Frasca, P. Longo, Interior $W^{2,p}%
$-estimates for nondivergence elliptic equations with discontinuous
coefficients, Ricerche Mat., 40 (1) (1991), 149-168.

\bibitem {ChFraL2}F. Chiarenza, M. Frasca, P. Longo,\ $W^{2,p}$-solvability of
Dirichlet problem for nondivergence elliptic equations with VMO coefficients,
Trans. Amer. Math. Soc., 336 (2) (1993), 841-853.

\bibitem {Coifman1}R.R. Coifman, R. Rochberg, G. Weiss,\ Factorization
theorems for Hardy spaces in several variables, Ann. Math., 103 (3) (1976), 611-635.

\bibitem {Coifman2}R.R. Coifman, P.L. Lions, Y. Meyer, S. Semmes, Compensated
compactness and Hardy spaces, J. Math. Pure Appl., 72 (3) (1993), 247-286.

\bibitem {FazPalRag}G. Di Fazio, D.K. Palagachev and M.A. Ragusa, Global
Morrey regularity of strong solutions to the Dirichlet problem for elliptic
equations with discontinuous coefficients, J. Funct. Anal., 166 (1999), 179-196.

\bibitem {Gurbuz}F. Gurbuz\textit{, }Boundedness of some potential type
sublinear operators and their commutators with rough kernels on generalized
local Morrey spaces $\left[  \text{\textit{Ph.D. thesis}}\right]  $, Ankara
University, Ankara, Turkey, 2015 (in Turkish).

\bibitem {Gurbuz1}F. Gurbuz, Parabolic sublinear operators with rough kernel
generated by parabolic Calder\'{o}n-Zygmund operators and parabolic local
Campanato space estimates for their commutators on the parabolic generalized
local Morrey spaces, Open Math., 14 (2016), 300-323.

\bibitem {Gurbuz2}F. Gurbuz\textit{, }Parabolic sublinear operators with rough
kernel generated by parabolic fractional integral operators and parabolic
local Campanato space estimates for their commutators on the parabolic
generalized local Morrey spaces, Adv. Math. (China), 2016, in press.

\bibitem {Gurbuz3}F. Gurbuz\textit{, }Sublinear operators with a rough kernel
generated by fractional integrals and local Campanato space estimates for
commutators with rough kernel on generalized local Morrey spaces, Int. J.
Appl. Math. \& Stat., 2016, in press.

\bibitem {Janson}S. Janson, Mean oscillation and commutators of singular
integral operators, Ark. Mat., 16 (1978), 263-270.

\bibitem {LLY}G. Lu, S.Z. Lu, D.C. Yang, Singular integrals and commutators on
homogeneous groups, Anal. Math., 28 (2002), 103-134.

\bibitem {LuYang1}S.Z. Lu and D.C. Yang, The central BMO spaces and
Littlewood-Paley operators, Approx. Theory Appl. (N.S.), 11 (1995), 72-94.

\bibitem {Miranda}C. Miranda,\textit{\ }Sulle equazioni ellittiche del secondo
ordine di tipo non variazionale, a coefficienti discontinui. Ann. Math. Pura E
Appl. 63 (4) (1963), 353-386.

\bibitem {Mazzucato}A. Mazzucato, Besov-Morrey spaces:functions space theory
and applications to non-linear PDE., Trans. Amer. Math. Soc., 355 (2002), 1297-1364.

\bibitem {Mo}H. Mo, H.Y. Xue, Commutators generated by singular integral
operators with variable kernels and local Campanato functions on generalized
local Morrey spaces, Commun. Math. Anal., 19 (2) (2016), 32-42.

\bibitem {Morrey}C.B. Morrey, On the solutions of quasi-linear elliptic
partial differential equations, Trans. Amer. Math. Soc., 43 (1938), 126-166.

\bibitem {Muckenhoupt}B. Muckenhoupt, R.L. Wheeden, Weighted norm inequalities
for singular and fractional integrals. Trans. Amer. Math. Soc., 161 (1971), 249-258.

\bibitem {Pal}D.K. Palagachev, L.G. Softova, Singular integral operators,
Morrey spaces and fine regularity of solutions to PDE's, Potential Anal., 20
(2004), 237-263.

\bibitem {Palus}M. Paluszynski, Characterization of the Besov spaces via the
commutator operator of Coifman, Rochberg and Weiss, Indiana Univ. Math. J., 44
(1995), 1-17.

\bibitem {PerRagSamWall}L.E. Persson, M.A. Ragusa, N. Samko, P. Wall,
Commutators of Hardy operators in vanishing Morrey spaces. AIP Conf. Proc.
1493, 859 (2012); http://dx.doi.org/10.1063/1.4765588.

\bibitem {RagusaJGlOpt}M.A. Ragusa, Commutators of fractional integral
operators on vanishing-Morrey spaces. J. Global Optim. 368 (40) (2008), 1-3.

\bibitem {Ruiz}A. Ruiz, L. Vega, On local regularity of Schr\"{o}dinger
equations, Int. Math. Res. Not., 1 (1993), 13-27.

\bibitem {Shi}S.G. Shi, S.Z. Lu, A characterization of Campanato space via
commutator of fractional integral, J. Math. Anal. Appl., 419 (2014), 123-137.

\bibitem {Softova}L.G. Softova, Singular integrals and commutators in
generalized Morrey spaces, Acta Math. Sin. (Engl. Ser.), 22 (2006), 757-766.

\bibitem {SW}F. Soria, G. Weiss, \ A remark on singular integrals and power
weights, Indiana Univ. Math. J., 43 (1994) 187-204.

\bibitem {Vitanza1}C. Vitanza, Functions with vanishing Morrey norm and
elliptic partial differential equations. In: Proceedings of Methods of Real
Analysis and Partial Differential Equations,Capri, pp. 147-150. Springer (1990).

\bibitem {Vitanza2}C. Vitanza, Regularity results for a class of elliptic
equations with coefficients in Morrey spaces. Ricerche Mat., 42 (2) (1993), 265-281.

\bibitem {Wang}H. Wang, Singular and fractional integral operators with
variable kernels on the weak Hardy spaces, Commun. Math. Anal., 18 (1) (2015), 48-63.

\bibitem {Wheeden-Zygmund}R.L. Wheeden and A. Zygmund, Measure and Integral:
An Introduction to Real Analysis, vol. 43 of Pure and Applied Mathematics,
Marcel Dekker, New York, NY, USA, 1977.

\bibitem {Zhang}P. Zhang and K. Zhao, Commutators of integral operators with
variable kernels on Hardy spaces, Proc. Indian Acad. Sci. (Math. Sci.), 115
(4) (2003), 399-410.
\end{thebibliography}
\end{document}